\documentclass[10pt,a4paper]{article}
\usepackage{amsmath}
\usepackage{amsfonts}
\usepackage{amssymb}
\usepackage{epsfig} 
\usepackage{latexsym}
\usepackage{geometry}
\usepackage{mathrsfs}
\usepackage{xcolor}
\usepackage{color}
\usepackage{hyperref}

\numberwithin{equation}{section}

\newtheorem{theorem}{Theorem}[section]

\newtheorem{definition}[theorem]{Definition}

\newtheorem{proposition}[theorem]{Proposition}
\newtheorem{lemma}[theorem]{Lemma}

\newtheorem{remark}{Remark}[section]

\newenvironment{proof}[2]{\noindent\textbf{Proof of #1 \ref{#2}}\newline}{\hfill$\square$}

\newcommand {\real} {{\mathbb R}}

\newcommand {\indicator} {{\mathbf{1} }}

\def\E{{\mathbb E}}

\def\P{\mathbb{P}}
\def\XX{\mathbb{X}}

\def\vtr{\ensuremath{V'_{trunc}}}
\def\utr{\ensuremath{u^h_{trunc}}}
\def\utrp{\ensuremath{u^{h,+}_{trunc}}}
\def\utrm{\ensuremath{u^{h,-}_{trunc}}}

\title{The scaling limit of a particle system with long-range interaction}
\author{Anton Bovier\thanks{Institut f\"ur Angewandte Mathematik, Universit\"at Bonn, Endenicher Allee 60, 
53115 Bonn, Germany.
email: \texttt{bovier@uni-bonn.de}} {} and 
Carina Geldhauser\thanks{Institut f\"ur Angewandte Mathematik, Universit\"at Bonn, Endenicher Allee 60, 
53115 Bonn, Germany.
email: \texttt{geldhauser@iam.uni-bonn.de}}
}

\begin{document}

\maketitle


\begin{abstract}
 We describe the macroscopic behaviour of a particle system with long-range interactions. We describe conditions on the interaction strength in dependency of the distance of the particles, such that the scaling limit of the particle system is a well-posed stochastic PDE. 
\end{abstract}

\section{Introduction}

Interacting particle systems model complex phenomena in natural and social sciences, such as traffic flow on highways or pedestrians, opinion dynamics, spread of epidemics or fires, reaction diffusion systems, crystal surface growth, chemotaxis and financial markets.
These phenomena involve a large number of interrelated components, which are modeled as particles confined to a lattice. Their motion and interaction is governed by local rules, plus some microscopic influences, which is modeled by an independent source of noise. Such noise can either be present in nature or it represents unresolved degrees of freedom.

Many models of particle systems are based on discrete on-site variables, so called spins. Our model, however, involves continuous local variables, and is therefore described by a system of interacting stochastic differential equations. Such models are sometimes called interacting diffusions.

In our model  \eqref{int:discretesystemunscaled}, each particle is subject to force derived from a bistable potential and perturbed by Brownian noise. The interaction between the particles is of long-range type, which means that each particle interacts with all particles which are at distance less or equal to $R$, where $R$ is very large, but significantly smaller than the total number of particles, which we denote by $N$.

Setting $R=1$ in our model gives the case of  nearest-neighbour interactions, whose dynamics has been studied in recent years by many authors. Due to the competition between local dynamics and coupling between different sites, a wide range of interesting behaviour was observed and investigated. Before stating  our model and the main result, we give a short overview on the history and some results on the nearest-neighbour case, which is very much related to our case:

Without noise, we know that there exist two stable states of the system. In presence of noise, the behaviour of the system is fundamentally different: Arbitrarily small random fluctuations can enable transitions between stable states at large time scales. Whether such transitions are observed will depend on the timescale of interest. The related concepts of phase transition, metastability and metastable timescales have been developed in the context of statistical-mechanics type models, for an overview see the recent book \cite{metabook} and the references therein.

For weak coupling, the behaviour is similar to the stochastic lattice models, where one often observes spatial chaos, i.e. independent dynamics at different sites. Also, bifurcations have been studied for the weak coupling regime, see \cite{gentzbf} and \cite{QC04} for a detailed analysis. 

However, as the coupling strength increases, the number of equilibrium points decreases. For strong coupling (of the order $N^2$, as in our case), the system synchronizes, in the sense that all particles assume almost the same position in their respective local potential most of the time.
 For large system size $N$, the behaviour of the nearest-neighbour interaction system is closer to the behaviour of a Ginzburg-Landau partial differential equation with noise, see \cite{eckmannhairer}, \cite{rougemont} and \cite{stein}, for example. Metastable behaviour in the large $N$ case (for $R=1$) has been studied in \cite{gentzpart2} and in \cite{BBM}, where sharp estimates on the metastable transition times between the two stable states have been obtained.

In the current work, we show that after suitable rescaling, our particle system \eqref{int:discretesystemunscaled} converges as $R$ and $N$ simultaneously go to infinity to the stochastic Allen Cahn equation \eqref{int:generalac} in one space dimension. Note that \eqref{int:generalac}  can be interpreted as a model of the movement of a random string, see \cite{funaki}. Existence and uniqueness of solutions to \eqref{int:generalac} has been proved in \cite{gyongypardoux}  via an approximation procedure similar to that in \cite{krylov}. Our proof via discretization in space is similar to the works \cite{gyongylattice} and \cite{gyongymillet}. 

The close relationship between our model and the nearest-neighbour model becomes obvious also in our main result: We obtain the same continuum limit as in the nearest-neighbour case (though of course the scaling is different), and the reason behind it are the assumptions we make on the interaction strength. We expect that there are weaker assumptions which lead to different, but still well-defined continuum limit,  this is subject of ongoing work.

\paragraph{Setting and statement of the main result}

We consider a system of $N$ coupled particles on a lattice $\Lambda = \mathbb{Z}/N \mathbb{Z} $. Each particle is subject to force derived from a bistable potential $V$ and perturbed by Brownian noise. 
The particle system can be described as a vector of initial positions $X^N(0) = (X_1^N, X_2^N, \ldots, X_N^N)$ and a system of $N$ coupled stochastic differential equations
\begin{equation}\label{int:discretesystemunscaled}
dX_i^N(t) = \frac{\gamma}{R^3} \sum_{j = - R}^R J_R(j)\left( X_i^N(t) - X_{i+j}^N(t) \right) dt - V'(X_i^N(t)) dt + \sqrt{2 \sigma} d\widetilde{B}_i(t), \; i \in \Lambda
\end{equation}
Here,  $X_i^N(t)$ are the components of the vector $X^N(t)\in \real^N$, $J_R(j) \in \real_+$ are weights, $V(q) = \frac{1}{4}q^4 - \frac{1}{2} q^2$ and  $\widetilde{B}_i$ are independent Brownian motions. $\gamma$ is a constant and $\sqrt{2 \sigma}$ the intensity of the noise.

\noindent In this model, each particle interacts with all of its neighbours up to distance $R$. The weights $ J_R(j)$ describes the strength of the interaction between two particles at site $i$ and $i+j$. 

We look for sufficient conditions on the weights $J_R(j)$ such that, after suitable rescaling, in the limit as $R,N \to \infty$, \eqref{int:discretesystem} gives rise to a well-posed stochastic PDE, 
\begin{equation}\begin{aligned}\label{int:generalac}
 \partial_t u (x,t) & =   \gamma A  u(x,t)  -  V'(u(x,t))  \; + \sqrt{2 \sigma}\frac{\partial^2}{\partial_x \partial_t}{W}(x,t) \qquad (x,t) \in  [0,1]  \times \real^+ \\
u(0,\cdot) & = u_0, \\
\end{aligned}
 \end{equation}
 where $A$ is the Laplace operator on $[0,1]$ with periodic boundary conditions, $\gamma > 0$ is the diffusion constant, $V$ a double well potential, $\frac{\partial^2}{\partial_x \partial_t}{W}(x,t)$ denotes space-time white noise and $\sqrt{2 \sigma}$ is the intensity of the noise.

 After a suitable rescaling of \eqref{int:discretesystemunscaled}, which we will present in detail in the next sections, we obtain the following result (see Theorem \ref{theo:asconvergence}, Theorem  \ref{theo:lpconvergence} and Theorem \ref{theo:transition}).
\begin{theorem}
Let $u^h_0$ be the piecewise linear approximation of an initial data $u_0 \in C^{4}$. Let $u(x,t)$ the solution to \eqref{int:generalac} and $u^h(x,t)$ 
the solution the following system of SDEs 
\begin{equation}\label{int:discretesystemintro}
du_i(t) =  \left(  \frac{\gamma}{R^{3} h^2} \sum_{j = - R}^R J_R(j)\left(  u_{i+j}(t) - u_i(t) \right)  \right)  dt - V'(u_i(t)) dt + \sqrt{\frac{2 \sigma}{h}} dB_i(t), \; i \in D_h
\end{equation}
with $J_R(j) = J\left(\frac{j}{R}\right)$, where $J$ is positive and satisfies $\int J(x) x^2 \, dx = 1$. 

\noindent If $R \sim h^{-\zeta}$ with $\zeta < \frac{1}{2}$, then 

i) for all times $T>0$, and all $p>1$, $u^h \longrightarrow u$ in $L^p\left(\Omega , C([0,1] \times [0,T])\right)$.

ii) for all times $T>0$, there exists an almost surely finite random variable $\XX$ such that
\begin{equation*}
 \sup_{[0,T] \times [0,1]} |u^h(x,t) - u(x,t) | \;  \leq \;  \XX  h^{\eta}  
\end{equation*}
for $0  < \eta  < \frac{1}{2} - \delta$. 

iii) Let $u_{\textup{min}}$ and $\tilde{u}_{\textup{min}}$ be the two minima of $V$. Let $u_0$ be close to $\tilde{u}_{\textup{min}}$. Define $ \tau (\rho , q) := \inf_{t >0} \left\lbrace  \|u - u_{\textup{min}}\|_{L^q([0,1])} < \rho \right\rbrace$ and $\tau^h (\rho , q) := \inf_{t >0} \left\lbrace \|u^h - u_{\textup{min}}^h\|_{L^q([0,1])} < \rho \right\rbrace$. Then we have for almost all $\rho >0$,
 \begin{equation*}
 \tau^h(\rho , q)  \longrightarrow \tau(\rho , q)  \qquad  \textup{a. s.  as } h \to 0
 \end{equation*}
 and 
  \begin{equation*}
 \E\left[\tau^h(\rho , q) \right] \longrightarrow  \E\left[ \tau(\rho , q) \right]   \qquad \textup{a. s.  as } h \to 0
 \end{equation*}
\end{theorem}
Result  $(iii)$ is the convergence of transition times of our discrete system to the transition times of the SPDE   \eqref{int:generalac}. Precise estimates on the transition times for \eqref{int:generalac} have been proved in \cite{barret} via a potential theoretic approach. Similar results for a more general class of one-dimensional parabolic stochastic partial differential equations, which include also the bifurcation cases, were obtained in \cite{gentz2013}.

\section{Properties of the discrete Operator}

\paragraph{Notation and rescaling}
We rescale the unit lattice $\Lambda = \mathbb{Z}/N \mathbb{Z}$ by $h=\frac{1}{N}$ to arrive at the uniform grid $\mathbb{T}_h = \lbrace 0, h, \ldots , Nh\rbrace $ where we identify $0=1=Nh$. $\mathbb{T}_h$ is then a discretization of the interval $[0,1]$ in equidistant nodes. We call $h$ the grid size and will sometimes refer to $ih$ as the ''node $i$''. 

Moreover, we rescale the coupling constant $\tilde{\gamma}$ by $h^{-1}$ and the potential term by $h$. Then we accelerate time by a factor $\frac{1}{h}$, i.e. we set $\widetilde{X}(t) = X(t/h)$, which gives us another extra $h^{-1}$ on the coupling constant and cancels out the previous changes in the scaling of the potential. Moreover, this acceleration of time gives us a different sequence of independent Brownian motions, which we call $B_i(t)$.
The real-valued stochastic process $\widetilde{X}^h_i(t)$ can then be identified with the real-valued function $u_i(t, \omega)$ of nodal values at the node $i$. 
The resulting rescaled system of SDEs reads
\begin{equation}\label{int:discretesystem}
du_i(t) =  \left(  \frac{\gamma}{R^{3} h^2} \sum_{j = - R}^R J_R(j)\left(  u_{i+j}(t) - u_i(t) \right)  \right)  dt - V'(u_i(t)) dt + \sqrt{\frac{2 \sigma}{h}} dB_i(t), \; i \in \mathbb{T}_h
\end{equation}
Note that $u_i(t)$ is defined only at one specific node $i$. Via $u_i(t) := u^h(ih, t)$, the vector-valued function of nodal values  $u^h(t) = (u_1(t), u_2(t), \ldots, u_N(t))$ on the grid $D_h$ can be identified as  a continuous, piecewise linear function on $u^h(x,t): D \times \real^+ \to \real$.

Note furthermore that we can relate the rescaled Brownian noise to space-time white noise via
\begin{equation}\label{int:noise}
  \sqrt{h} B_i(t) = \int_{(i-1)h}^{ih} W(x,t) \, dx 
\end{equation}
We can rewrite \eqref{int:discretesystem} in integral form as
\begin{equation}\label{int:discreteintegralform}
\begin{aligned}
u^h(x,t) \, = \; \int_0^1 g^h_t(x,y) u^h_0(y) \, dy  \; - \; \int_0^t \int_0^1 g^h_{t-s}(x,y) V'((u^h(y,s)) \, dy ds \\
+ \sqrt{2 \sigma } \int_0^t \int_0^1 g^h_{t-s}(x,y) \, W( dy, ds)
\end{aligned}
\end{equation}
where $g^h_t(x,y)$ is the semigroup associated with the discrete operator $- \gamma A^h_R$ defined in \eqref{int:Rpointstar}. $g^h_t(x,y)$ is defined on $\real^+ \times [0,1] \times [0,1]$ in a piecewise linear fashion, see \eqref{int:discreteSG}.

\paragraph{Central difference operator and weights}

We interpret the collection of weights $J_R(j)$ as a weight function $ J_R: \mathbb{Z} \to \real_+$,  $J_R(j) = J\left(\frac{j}{R}\right)$ for some positive function $J$ satisfying 
\begin{equation}\label{int:J}
 \int_{\real} x^2 J(x) dx \; = \; 1
\end{equation}
Typical examples are $J_R(j) = c \exp(-j/R)$ or $J_R(j) = c \indicator_{|j|\leq R}$. 
The weights $J_R(j)$ are the entries of the $j$-th subdiagonal of the band matrix $A^h_R$, where $R$ indicates the width of the stencil:
\begin{equation}\label{int:Rpointstar}
A^h_R u_i = \frac{J_R(R) u_{i+R} + \ldots + J_R(1) u_{i+1}   - J_R(0) u_i + J_R(-1) u_{i-1} \ldots J_R(-R) u_{i-R}  }{R^{3} h^2}
\end{equation} 
Note that the weights $J_R(j)$ are fixed and do not change with time, so the central difference operator $-A^h_R$ is time-independent.
Moreover, $-A^h_R$ is a positive definite matrix, as, by construction from the model \eqref{int:discretesystem}, the values of the weight function satisfy the diagonal dominance relation $J_R(0) = 2 \sum_{j \geq 1} J_R(j) \;  = \; 2 \sum_{j=1}^R J_R(j)$ where $J_R(0)$ is the weight attributed to the reference site $i$.

\paragraph{Boundedness of the inverse}
The big difference between the particle system with long-range interaction and a particle system with nearest-neighbour interaction is that the interaction length $R$ actually tends to infinity as the number of particles go to infinity. 

We have modelled our particle system as a discretization in the space variable of a continuous limit, which means that instead of discussing the limit as the number of particles $N$ go to infinity, we actually consider the limit $h\to 0$ of a semidiscrete finite difference scheme (with $h = \frac{1}{N}$). As we consider the simultaneous limit of both variables $h$ and $R$, it is convenient to rewrite $R$ in terms of $h$, so we define $R = h^{-\zeta}$ with  $0 <  \zeta < 1$.
In the next lemma  we derive the admissible values of $\zeta$ such that $(-A^h_R)^{-1}$ is a bounded operator:
\begin{lemma}\label{lemma:J1}
 Let 
 \begin{equation}\label{int:Rstar}
 \begin{aligned}
 \gamma A_R^h  u_i &= \frac{\gamma}{R^{3} h^2} \sum_{j = - R}^R J_R(j)\left(  u_{i+j}(t) - u_i(t) \right) 
   \end{aligned}
 \end{equation}
with $ R \sim h^{-\zeta}$, and  $J_R(j)$ satisfying \eqref{int:J}. Let the eigenvalues of \eqref{int:Rstar} with periodic boundary conditions be denoted by  $\lambda_k^h = \frac{4\gamma }{h^2 R^3} \sum_{j=1}^R J_R(j) \sin^2\left(\frac{\pi}{2}khj  \right)$. 

Then, for $ \zeta < \frac{1}{2}$ 
 \begin{equation}\label{int:scalingJ1}
\begin{aligned}
\sum_{k=1}^{1/h}  \left(\lambda_k^h\right)^{-1} \; \leq \;  \sum_{k=1}^{h^{\zeta -1} }  \frac{1}{c k^2} + o(h^{1- 2 \zeta}) \; < \; \infty
\end{aligned}
\end{equation}
\end{lemma}
\begin{proof}{Lemma}{lemma:J1}
Consider first $J_R(j) = c \indicator_{|j|\leq R}$. 
We need to show that 
\begin{equation}\label{int:summabilityInd1}
\begin{aligned}
\sum_{k=1}^{1/h}  \left(\lambda_k^h\right)^{-1} \; \leq \; \sum_{k=1}^{1/h} \frac{1}{c} \frac{h^{-3 \zeta +2}}{\sum_{j=1}^{h^{-\zeta}} \sin^2\left(\frac{\pi}{2}khj\right) } \; < \infty
\end{aligned}
\end{equation}
We split the sum in $k \leq h^{\zeta -1}$ and $k > h^{\zeta -1}$. As for $k \leq h^{\zeta -1}$ we have $khj \leq 1$, the increment of the sine squared stays inside the regime $\left[0, \frac{\pi}{2}\right]$.
In this regime, we use the fact that $\sin^2(x) \geq \frac{4x^2}{\pi^2}$ to estimate 
\begin{equation}\label{int:summabilityInd2}
\begin{aligned}
\sum_{k=1}^{h^{\zeta-1}} \frac{1}{c} \frac{h^{-3 \zeta +2}}{\sum_{j=1}^{h^{-\zeta}} \sin^2\left(\frac{\pi}{2}khj\right) } \; 
\leq \; \sum_{k=1}^{h^{\zeta -1}} \frac{1}{c} \frac{h^{-3 \zeta +2} }{\sum_{j=1}^{h^{-\zeta}} k^2 h^2 j^2 } \;
\leq \; \sum_{k=1}^{h^{\zeta -1}} \frac{1}{c k^2}  \\
\end{aligned}
\end{equation}
which gives the first term in \eqref{int:scalingJ1}.

For $k > h^{\zeta -1}$, note that for any $j$ there exists a $k(j)$ such that $khj = 1$, for which we have $\sin^2\left(\frac{\pi}{2}khj\right) = 1$. Therefore, the denominator $\sum_{j=1}^{h^{-\zeta}} \sin^2\left(\frac{\pi}{2}khj\right)$ can be bounded from below by this element, which has the value $\sin^2\left(1 \right) = 1 $. Moreover, note that in the regime $\frac{h^{-\zeta}}{2} \leq j \leq h^{-\zeta}$, the denominator is bounded from below by $\sin^2\left(\frac{1}{2} \right)$. This gives with $\zeta < \frac{1}{2}$
\begin{equation}\label{int:summabilityInd4}
\begin{aligned}
\sum_{k > h^{\zeta -1}}^{1/h} \frac{1}{c} \frac{h^{- 3 \zeta +2}}{\sum_{j=1}^{h^{-\zeta}} \sin^2\left(\frac{\pi}{2}khj\right) }
\; &\leq \; \frac{1}{c} \left( h^{- 2 \zeta +2} \cdot h^{-1} \right) \; = \; o( h^{- 2 \zeta +1}) 
\end{aligned}
\end{equation}
\end{proof}

A direct consequence of Lemma \ref{lemma:J1} is the identity
 \begin{equation}\label{int:extraass1}
\frac{c}{R^3} \sum_{j=1}^R J_R(j) j^2 = 1
 \end{equation}
 which is a discrete version of the second moment condition on $J$. Moreover, we get
 \begin{equation}\label{int:extraass1b}
\frac{c}{R^3}  \sum_{j=1}^R J_R(j) j^4 = o( h^{- 2\zeta } ).
 \end{equation}

\paragraph{Convergence of eigenvalues and eigenvectors}
In this section, we state some useful facts on the eigenvalues and eigenvectors of  the long-range discrete operator $- \gamma  A^h_R$. 

\textbf{Notation:} From now on, when writing $- \gamma  A^h_R$ we mean the discrete symmetric stencil with appropriate choices of coefficient and scaling as stated in \eqref{int:Rstar}.

The eigenvalues  of $- \gamma  A^h_R$ with periodic boundary conditions read
\begin{equation}\label{int:eigenvalueslongrange}
\lambda_k^h = \frac{4\gamma }{h^2 R^3} \sum_{j=1}^R J_R(j) \sin^2\left(\frac{\pi}{2}khj  \right)
\end{equation}
Property \eqref{int:extraass1} gives immediately the upper bound
\begin{equation}\label{int:upperew}
 \begin{aligned}
  \lambda_k^h \leq  \gamma \pi^2  k^2 = \lambda_k \\
 \end{aligned}
\end{equation}
and as for $0 \leq x \leq \frac{\pi}{2}$, $\sin^2(x) \geq \frac{4x^2}{\pi^2}$, the following lower bound holds
\begin{equation}\label{int:lbew}
\begin{aligned}
\lambda_k^h &\geq \; \frac{4\gamma k^2}{R^3} \sum_{j=1}^R J_R(j) j^2  \overset{\eqref{int:extraass1}}{=} \; 4 \gamma k^2
\end{aligned}
\end{equation}
We will use this inequality frequently for estimates on the discrete semigroup.

Thanks to Lemma \ref{lemma:J1} we can conclude convergence of the eigenvalues:
\begin{lemma}\label{lemma:ewconvergence}
 Let $\lambda_k$ 
 be the eigenvalues of $A$ on $[0,1]$ and $\lambda^h_k$ given by \eqref{int:eigenvalueslongrange}  the eigenvalues of the the long-range discrete operator $- \gamma  A^h_R$. Let $R \sim h^{-\zeta}$ and $J_R(j)$ as in Lemma \ref{lemma:J1}.
Then we have 
\begin{equation}
\begin{aligned}
\lambda_k - \lambda_k^h \longrightarrow \; 0 \quad \textup{ as } h \to 0
\end{aligned}
\end{equation}
with rate up $h^{2 - 2\zeta}$.
\end{lemma}
\begin{proof}{Lemma}{lemma:ewconvergence}
We use the fact that for $0 \leq x \leq \frac{\pi}{2}$, $\sin^2(x) \geq x^2 ( 1- \frac{x^2}{3})$: 
\begin{equation}\label{int:ewconvergence}
\begin{aligned}
\lambda_k^h 
&\geq \frac{\gamma \pi^2 k^2 }{R^3} \sum_{j=1}^R \left\lbrace J_R(j) j^2  \left( 1 - \frac{1}{12}k^2h^2\pi^2 j^2\right)\right\rbrace \\
&\geq c \gamma \pi^2 k^2 \underbrace{\frac{1}{R^3} \sum_{j=1}^R J_R(j) j^2}_{\overset{\eqref{int:extraass1}}{=} 1} - \frac{\gamma}{12}k^4\pi^4 h^2  \underbrace{\frac{1}{R^3} \sum_{j=1}^R J_R(j) j^4}_{\overset{\eqref{int:extraass1b}}{=} h^{-2\zeta}}\\
&\leq \frac{\gamma}{12}  k^4 \pi^4 \cdot h^{2 - 2\zeta} = c(\gamma) h^{2 - 2\zeta}
\end{aligned}
\end{equation}
Note that the rate of convergence depends on $\zeta$. 
As $\zeta < \frac{1}{2}$, the rate $h^{2 - 2\zeta}$ is at most $ h^{1 + \epsilon}$.
\end{proof}
\paragraph{Eigenvectors}
The $m$-th entry of the eigenvector $v^h_k$ (the eigenvector of $- \gamma  A^h_R$ associated with the eigenvalue $\lambda^h_k$) reads:
\begin{equation}\label{int:eigenvector}
v^h_k(m) = \sin (\pi k mh) 
\end{equation}
Obviously, we get the following result:
\begin{lemma}\label{lemma:evconvergence}
Given $v_k = \sin (\pi k x)$ the eigenfunction of the Laplace operator with periodic boundary conditions on $[0,1]$ assiociated to the eigenvalue $\lambda_k$ and $v^h_k = (v^h_k(1), \ldots v^h_k(N))$ with $v^h_k(i)$ as in \eqref{int:eigenvector} the eigenvectors of $- \gamma A^h_R$.
Given $x\in [h(m-\frac{1}{2}), h(m+\frac{1}{2})[$.
Then, as $h\to 0$,
\begin{equation}\label{int:evconvergence}
\begin{aligned}
| v_k(x) - v^h_k(m) | \longrightarrow 0 \qquad \textup{as } h \to 0
 \end{aligned}
\end{equation}
\end{lemma}
\hfill$\square$

\paragraph{Consistency of a difference operator}

Assuming sufficient regularity of the solution up to the boundary of the domain, to ensure convergence of a finite difference operator, we need that it is \textit{consistent}, which means a vanishing local error as the grid size goes to zero.
The order of consistency tells us about the rate of convergence of a difference stencil to a continuous operator. It is derived using the Taylor formula and comparing the coefficients. 
This approach leads to high regularity restrictions such as  $u\in C^{m+2}$, see for example \cite{nick} for details.

Exploiting cancellation effects given by the symmetry of the stencil and the equidistant grid, we get for $A^h_R$ as in \eqref{int:Rstar} with initial data $u_0 \in C^4([0,1])$
\begin{equation}\label{int:consistency}
\begin{aligned}
\sup_{y\in [0,1]} \big| A^h_R u(y) -  u_{xx}(y)\big| \leq O(h^2) \\
\end{aligned}
\end{equation}

\subsection{Convergence of the discrete semigroup}

Let $-\gamma A^h_R$ as in in \eqref{int:Rstar}. Denote by $\lambda_k^h$ the eigenvalues of $-\gamma A^h_R$ stated in \eqref{int:eigenvalueslongrange} and by $v_k^h$ the piecewise linear functions on the grid realized by the  eigenvectors of  $-\gamma A^h_R$, see  \eqref{int:eigenvector}. The discrete semigroup associated to $-\gamma A^h_R$ reads
 \begin{equation}\label{int:discreteSG}
  g_t^h(x,y) = \sum_{k=1}^{1/h} e^{-t \lambda_k^h} v_k^h(x) v_k^h(y)
 \end{equation}
Thanks to Lemma \ref{lemma:ewconvergence} and \ref{lemma:evconvergence}, we can already prove uniform convergence of $g^h$ to $g$ on  $[t_0,\infty) \times [0,1]^2$:
\begin{proposition}\label{prop:sgconv}
Let $ g_t(x,y)$ as in \eqref{int:sg} and $ g^h_t(x,y)$ as in \eqref{int:discreteSG} with eigenvalues as stated in \eqref{int:eigenvalueslongrange}  and eigenvectors \eqref{int:eigenvector}.
 For all $t_0 > 0$ and $\zeta < \frac{1}{2}$, there exists a constant $c(\gamma, t_0 )$ such that for all $(t,x,y) \in [t_0,\infty) \times [0,1]^2$
 \begin{equation}
 |g_t^h(x,y) - g_t(x,y)| \leq c(\gamma , t_0) h^{2-2\zeta} 
 \end{equation}
\end{proposition}

\begin{proof}{Proposition}{prop:sgconv}
We first look at the difference $g^h(x,y) - g(x,y)$ for fixed $k$, where we employ the convergence of the eigenvalues for $\zeta < \frac{1}{2}$ to get
\begin{equation}\label{int:sgfixedk}
 \begin{aligned}
|g_t^h(x,y) - g_t(x,y)| & \leq 2 e^{- \lambda^h_k t} | 1 - e^{-t (\lambda_k -  \lambda^h_k)}| + 2 e^{- \lambda_k t}  \big| v^h_k(x)- v_k(x)\big|
 \end{aligned}
\end{equation}
Now we sum over all terms, use the above estimate and calculate
\begin{equation}\label{int:sgsumconvkh}
 \begin{aligned}
  |g_t^h(x,y) - g_t(x,y)| &\leq \sum_{k=1}^{1/h }\left|e^{- \lambda^h_k t} v^h_k(x) v^h_k(y) - e^{- \lambda_k t} v_k(x) v_k(y) )\right|  + \underbrace{\sum_{k > 1/h}^{\infty} e^{- \lambda_k t} v_k(x) v_k(y)}_{(\ast)} \\
    &\overset{\eqref{int:lbew}}{ \leq } 2 \sum_{k=1}^{1/h }   \left( e^{- 4 \gamma k^2 t} \left| 1 - e^{-t \gamma k^4 \frac{\pi^4}{12}  h^{2-2\zeta}} \right|  +  \sqrt{2} \pi k h  \cdot e^{-  \gamma \pi^2 k^2 t} \right)  + (\ast)  \\
        &\leq   \frac{\pi^4}{6}  \gamma t    h^{2-2\zeta}  \sum_{k=1}^{1/h }   k^4   e^{- 4 \gamma \pi k^2 t}  \;  + \; 2  \sqrt{2} \pi \;\sum_{k=1}^{1/h }   k h  \cdot e^{-  \gamma \pi^2 k^2 t} + (\ast)  \\
        &\leq   h^{2-2\zeta}   \frac{ \gamma t \pi^4 }{12(4\pi \gamma t)^{5/2}} \Gamma \left(\frac{5}{2}\right) +   h  \cdot  \frac{\sqrt{2} }{ \pi \gamma t}  + \frac{h^2}{(\gamma \pi^2 t)^{\frac{3}{2}}} \Gamma \left( \frac{3}{2}\right)
     \end{aligned}
\end{equation}
this gives
\begin{equation}
 \begin{aligned}
 \sum_{k=1}^{1/h }  \left|e^{- \lambda^h_k t} v^h_k(x) v^h_k(y) - e^{- \lambda_k t} v_k(x) v_k(y) )\right|  &\leq  c(\gamma , t) \left(  h^{2-2\zeta}  +  h + h^2 \right)
 \end{aligned}
\end{equation}
\end{proof}

\section{Estimates on the stochastic integral}

\subsection{$L^2$ estimates on the discrete semigroup}
\begin{lemma}\label{lemma:discSG3}
For the discrete semigroup \eqref{int:discreteSG} the following $L^2$ estimates hold:
\begin{equation}\label{int:dsg1}
 \int_0^1 g_t^h(x,y)^2 \, dy \; \leq \; c (\gamma , t)
\end{equation}
\begin{equation}\label{int:dsg2}
  \int_0^1\int_0^1 g_t^h(x,y)^2 \,  dy \, dx \; \leq \;  c (\gamma , t)
\end{equation}
\begin{equation}\label{int:dsg3}
  \int_0^t \int_0^1 g_s^h(x,y)^2 \, dy \, ds \; \leq \; c (\gamma , t)
\end{equation}
\begin{equation}\label{int:dsg4}
  \int_0^t \int_0^1\int_0^1 g_s^h(x,y)^2 \, dy \, dx \, ds \; \leq \; c (\gamma , t)
\end{equation}
\end{lemma}
\begin{proof}{Lemma}{lemma:discSG3}
To derive \eqref{int:dsg1}, we estimate
 \begin{equation}
 \begin{aligned}
 \sum_{k=1}^{1/h} e^{-2t \lambda_k^h} v_k^h(x)^2 
   &\overset{\eqref{int:lbew} }{\leq }  2 \sum_{k=1}^{1/h} e^{- 8 \gamma k^2 t } \leq \frac{\sqrt{\pi}}{2 \sqrt{2} \sqrt{\gamma t }}  \; \leq \;  c(\gamma , t) 
 \end{aligned}
 \end{equation}
The second inequality \eqref{int:dsg2} follows from the first by integration, which does not change the bound.
To derive \eqref{int:dsg3}, we estimate
 \begin{equation}
 \begin{aligned}
 \int_0^t \int_0^1 g_s^h(x,y)^2 dy ds
  &\overset{\eqref{int:lbew} }{\leq } \frac{1 }{ 8 \gamma} \sum_{k=1}^{1/h} \frac{1 }{k^2 } \left( 1 - e^{- 8 \gamma k^2 t} \right) 
 \leq  \frac{3}{ 8 \gamma } \min \left\lbrace \sqrt{8 \gamma t }, 1 \right\rbrace \, 
 \leq c (\gamma , t) 
 \end{aligned}
 \end{equation}
The fourth inequality \eqref{int:dsg4} follows from the first by integration, which does not change the bound.
\end{proof}

\subsection{Regularity of the discrete semigroup}
\begin{lemma}\label{lemma:discSG4}
Let $g_t^h(x,y)$ be the discrete semigroup in \eqref{int:discreteSG} and 
let $(x,t), (x',t')\in D \times [0,T]$ with $t' > t$. 
Then we have the following estimates:
\begin{equation}\label{int:spacevarsg}
 \int_0^t \int_0^1\big| g_{t-s}^h(x,y)  - g_{t-s}^h(x',y) \big|^2 dy ds \; \leq \; c(\gamma) |x - x'|
\end{equation}
and in time
\begin{equation}\label{int:timevarsg}
 \int_0^t \int_0^1\big| g_{t-s}^h(x,y)  - g_{t'-s}^h(x,y) \big|^2 dy ds \; \leq \; c(\gamma) \sqrt{|t - t'|}
\end{equation}
\end{lemma}
\begin{proof}{Lemma}{lemma:discSG4}
\underline{Part 1: Proof of \eqref{int:spacevarsg}}
Take two grid points $x = ih$ and $x' = mh$. As $|v^h_k(y)|^2 = 2 $ and $\int_0^t \exp(-2 \lambda_k^h  (t-s)) \leq \frac{1}{2\lambda_k^h}$, we get
\begin{equation}
 \begin{aligned}
 \int_0^t \int_0^1\big| g_{t-s}^h(x,y)  - g_{t-s}^h(x',y) \big|^2 dy 
  &\overset{\eqref{int:lbew} }{\leq } \sum_{k=1}^{1/h}  \frac{1}{8 \gamma k^2 } \pi k |x-x'|^2 
   &= c(\gamma) |x - x'|
  \end{aligned}
\end{equation}
\underline{Part 2: Proof of \eqref{int:timevarsg}}
By orthogonality of the eigenvectors, we can estimate
\begin{equation}
 \begin{aligned}
 \int_0^t \int_0^1\big| g_{t-s}^h(x,y)  - g_{t'-s}^h(x,y) \big|^2 dy
 &\leq \, 2 \,  \sum_{k=1}^{1/h} \frac{1}{2\lambda_k^h} \big|1 - e^{-(t'-t)\lambda_k^h}\big|^2  \\
  &\overset{\eqref{int:lbew} }{\leq } \, 2 \,  \sum_{k=1}^{1/h} \frac{1}{8 \gamma k^2 }  \big|1 - e^{-(t'-t)\lambda_k^h}\big|^2 \\
  &\overset{\eqref{int:upperew}}{\leq} c(\gamma) \sqrt{|t - t'|}
  \end{aligned}
\end{equation}
\end{proof}

\subsection{Regularity of the discrete stochastic integral}

\begin{lemma}\label{lemma:discSG5b}
Given the discrete semigroup \eqref{int:discreteSG} and  a sequence $u^h$ of random variables which satisfy
\begin{equation}\label{int:uhassS}
 \sup_h \sup_{t\in [0,T]} \sup_{x\in [0,1]} \E[ |u^h(x,t)|^p] \leq  C
\end{equation}
Define the stochastic integral  
\begin{equation}
 S(x,t) = \int_0^t \int_0^1 g^h_{t-s} (x,y) u^h(y,s) \; W(dy,ds)
\end{equation}
Then we have for  $1 \leq p < \infty$ and $T>0$ 
\begin{equation}\label{int:stochintreg1}
 \E\Big[ \Big|  S(x,t)  -  S(z,\tilde{t})   \Big|^p\Big] \leq \; c(p,T) \left( |t-\tilde{t}|^{\frac{1}{4}} +  |x-z|^{\frac{1}{2}}\right)^p
\end{equation}
with $c(p,T)$ independent of $h$. 
In particular, for $(x,t), (z,\tilde{t}) \in [0,1] \times [0,T]$ and  some exponent $\delta < \frac{1}{4}$ we have the following H\"older regularity estimate
\begin{equation}\label{int:stochintreg}
\Big|  S(x,t)  -  S(z,\tilde{t})   \Big|^p \leq \; Y(p,T, \delta, h) \left( |t-\tilde{t}|^{\frac{1}{4} - \delta } +  |x-z|^{\frac{1}{2} - \delta}\right)
\end{equation}
where $Y(p,T, \delta, h)$ is a random variable in $L^p$ with moment bound independent of $h$:
\begin{equation}
 \E\Big[ Y(p,T, \delta, h)^{\frac{1}{\delta}}\Big]  \; \leq \; C(p,T,\delta)
\end{equation}
\end{lemma}

\begin{proof}{Lemma}{lemma:discSG5b}
We look at variations in the space and time variable separately.
For the variation in space, we employ the Burkholder-Davis-Gundy inequality and Lemma \ref{lemma:discSG4} to get 
\begin{equation}
\begin{aligned}
\E & \Big[ \Big| S(z ,t)  -  S(y,s)  \Big|^{2p} \Big]^{\frac{1}{p}} \leq c(\gamma) |x-z| \sup_{[0,1] \times [0,T]} \E\Big[  u^h(y,s)^{2p} \Big]^{\frac{1}{p}}  \\
 \end{aligned}
\end{equation}
using Assumption \eqref{int:uhassS} and taking the $p/2$th power, we arrive at the first part of \eqref{int:stochintreg1}.
Similarly, for the variation in time, with $\tilde{t} = t+r$, we get by BDG
\begin{equation}
\begin{aligned}
 \Big| & S(x,t+r )  -  S(x ,t)   \Big|^{2p} \; \\
&\leq \E \Big[ \Big|\int_0^{t+r } \int_0^1 | g^h_{t+r -s} (x,y) \, - \,  g^h_{t-s} (x,y) |^2 \; | u^h(y,s) |^2 \; dy ds \Big|^p \Big]^{\frac{1}{p}}  \\ 
& \leq \int_0^{t} \int_0^1 \Big| g^h_{t+r -s} (x,y) \, - \,  g^h_{t-s} (x,y) \Big|^2 \; \E  \Big[  | u^h(y,s) |^{2p}  \Big]^{\frac{1}{p}}  \; dy ds \\
& + \int_t^{t+r }  \int_0^1 \Big| g^h_{t+r -s} (x,y) \, - \,  g^h_{t-s} (x,y) \Big|^2 \; \E  \Big[  | u^h(y,s) |^{2p}  \Big]^{\frac{1}{p}}  \; dy ds \\
\overset{\eqref{int:dsg3}}{\leq}& c(\gamma ,t) \sqrt{| t - \tilde{t}|}  \sup_{[0,1] \times [0,T]}  \E  \Big[  | u^h(y,s) |^{2p}  \Big]^{\frac{1}{p}} 
 \end{aligned}
\end{equation}
Assumption \eqref{int:uhassS} and taking the $p/2$th power gives the second part of \eqref{int:stochintreg1}, which concludes the proof.
The estimate \eqref{int:stochintreg} follows from \eqref{int:stochintreg1} by direct application of the Kolmogorov-Centsov theorem.
\end{proof}

\begin{lemma}\label{lemma:discSG6}
The random variable
\begin{equation}\label{int:Bh}
 B^h(x,t) = \int_0^t \int_0^1 g^h_{t-s} (x,y) \; W(dy,ds)
\end{equation}
with $g^h_{t-s} (x,y)$ defined as in \eqref{int:discreteSG} is continuous on $[0,1] \times \real^+$. 

Moreover, for $T>0$ and some exponent $\delta < \frac{1}{4}$ there exists a random variable $Y^h(p,T, \delta )$ in $L^p$ such that for all $(x,t)$ and $(z,\tilde{t})$ in $[0,1] \times [0,T]$ the following inequality holds
\begin{equation}\label{int:holderstatement}
\Big|  B(x,t)  -  B(z,\tilde{t})  \Big| \leq \; Y^h (p,T, \delta) \left( |t-\tilde{t}|^{\frac{1}{4} - \delta } +  |x-z|^{\frac{1}{2} - \delta}\right)
\end{equation}
where 
\begin{equation}\label{int:yhass}
 \sup_h \E\Big[ Y_h^p\Big]  \; \leq \; C(p,T,\delta)
\end{equation}
\end{lemma}
with a constant which is independent of $h$. 
Moreover, we have for all  $1 \leq p < \infty$:
\begin{equation}\label{int:supEBh}
 \sup_h \E \Big[ \sup_{[0,1] \times [0,T]} \Big| \int_0^t \int_0^1 g^h_{t-s} (x,y) \; W(dy,ds)\Big|^p \Big] \; \leq C(p,\gamma , T) 
\end{equation}

\begin{proof}{Lemma}{lemma:discSG6}
The H\"older continuity statement \eqref{int:holderstatement} follows directly from estimate \eqref{int:stochintreg} of Lemma \ref{lemma:discSG5b}. For the second statement \eqref{int:supEBh},  we discretize time into $m$ intervals of size $\frac{T}{m}$ and call the time nodes $t_i := \frac{iT}{m}$. We choose for convenience the spatial grid point $z=\frac{1}{2}$ in the time-discrete stochastic integral and calculate
\begin{equation}
\begin{aligned}
 \sup_{[0,1] \times [0,T]} \Big| B^h(x,t) \Big| 
  &\overset{\eqref{int:holderstatement} }{\leq} \max_{i=1, \ldots, m}  \Big| B^h\left(\frac{1}{2}, t_i\right)\Big| + Y^h \left( \left(\frac{1}{2}\right)^{\frac{1}{2}-\delta} + \left( \frac{T}{m}\right)^{\frac{1}{4}-\delta}\right) \\
    &\leq \max_{i=1, \ldots, m}  \Big| B^h\left(\frac{1}{2}, t_i\right)\Big| + c(p,T,\delta) Y^h  \\
 \end{aligned}
\end{equation}
As $m$ is finite, we conclude with Burkholder-Davis-Gundy
\begin{equation}
\begin{aligned}
\sup_h  \E\Big[ \sup_{[0,1] \times [0,T]} \Big| B^h(x,t) \Big|^p \Big] 
&\leq \sup_h  \left( 2^p \E\Big[ \max_{i=1, \ldots, m}  \Big| B^h\left(\frac{1}{2}, t_i\right)\Big|^p\Big]  + c(p,T,\delta) \E\Big[Y_h^p\Big]\right) \\
& \overset{\eqref{int:yhass}}{\leq } 2^p \sup_h \sum_{i=1}^{m} \E\Big[ \Big| B^h\left(\frac{1}{2}, t_i\right)\Big|^p\Big]  + c(p,T,\delta) \\
& \overset{BDG, \eqref{int:dsg3} }{\leq } c(p,\gamma, \delta, T)\\
\end{aligned}
\end{equation}
\end{proof}

\section{A priori estimates}
In this section, we prove a bound on the moments of our discrete solution, which is independent of the grid size $h$. As the nonlinearity $V'$ is not lipschitz continuous, we will first truncate the discrete solution and prove a moment bound on the truncated solutions $\utr$. Using a comparison principle, we can then control the discrete solutions and infer a moment bound on the nontruncated discrete solutions.

\subsection{Existence of mild solutions}

\paragraph{Existence of mild solutions to the continuous equation}
Let $D\subset \real$ be a bounded interval and $x\in D$. Given any bounded continuous initial condition $u_0$, the Green's function of  the heat equation can be expressed as 
\begin{equation}\label{int:sg}
g_t(x,y)\;  = \; \sum_{k=1}^{\infty} e^{-\lambda_k t} v_k(x) v_k(y)
\end{equation}
where $\lambda_k = \gamma \pi^2 k^2$ are the eigenvalues and $v_k(x) = \sin(\pi k x)$ the eigenfunctions of the Laplacian with the corresponding boundary conditions.
\begin{definition}
A random field $u$ is called \textit{mild solution } to the equation \eqref{int:generalac} if (i) $u$ is almost surely continuous, measurable, and (ii) satisfies for all $(x,t) \in D \times \real^+$
\begin{equation}
\begin{aligned}
u(x,t) \, = \; \int_D g_t(x,y) u_0(y) \, dy \; - \; \int_0^t \int_D g_{t-s}(x,y) V'(u(y,s)) \, dy ds \\
+ \sqrt{2 \sigma } \int_0^t \int_D g_{t-s}(x,y) \, W( dy, ds)
\end{aligned}
\end{equation}
\end{definition}
We recall the following existence result  by Gy\"ongy and Pardoux \cite{gyongypardoux}.
\begin{proposition}\label{prop:continuoussolution}
For every initial conditions $u_0 \in C([0,1])$, the SPDE \eqref{int:generalac}
admits a unique mild solution. 
Moreover, for all times $T >0$ and $p \geq 1$,
\begin{equation}\label{int:continousmoments}
 \E \left[ \sup_{[0,T] \times [0,1]} |u(x,t)|^p \right] \leq C(T,p)
\end{equation}
The random field $u$ is $2\alpha$-H\"older in space and $\alpha$-H\"older in time for every $\alpha \in \left(0,\frac{1}{4}\right)$.
\end{proposition}
\hfill$\square$
\paragraph{Existence of mild solutions to discrete system}
\begin{proposition}\label{prop:discretesolution}
Given a suitable deterministic initial condition $u_0 \in C^4([0,1])$ with piecewise linear approximation $u^h(0)$, the rescaled system of SDEs \eqref{int:discretesystem}
admits a unique mild solution for all times $T$.
\end{proposition}
As $-V(u) \cdot u \leq C(h) \left(1+ \|u\|^2_{L^2}\right)$, and $h$ is fixed, the result follows by classical SDE theory. \hfill$\square$

\subsection{Uniform bound on moments of truncated solutions}
We start with the bound on the moments of truncated solutions, which is separated in the $\sup \E$ and $\E \sup$-bounds part.

Define the truncated drift
\begin{equation}\label{int:defvtr}
 \vtr (u) = V'(u) \indicator_{[-Z,Z]} + V'(Z) \indicator_{[Z, \infty)} + V'(-Z) \indicator_{(-\infty , -Z)} 
\end{equation}
which is a bounded and globally lipschitz function. In particular, 
\begin{equation}\label{int:VboundM}
\vtr(u) \leq V'(Z) =: M 
\end{equation}

\paragraph{Equations with truncated drift}

A mild solution to \eqref{int:generalac} with nonlinearity $V$ replaced by $\vtr$ will be denoted by $u_Z$ and reads
\begin{equation}\label{int:uz}
\begin{aligned}
u_Z(x,t) \, = \; \int_o^1 g_t(x,y) u_0(y) \, dy \; - \; \int_0^t \int_0^1 g_{t-s}(x,y) \vtr (u(y,s)) \, dy ds \\
+ \sqrt{2 \sigma } \int_0^t \int_0^1 g_{t-s}(x,y) \, W( dy, ds)
\end{aligned}
\end{equation}
Similarly, a mild solution to \eqref{int:discreteintegralform} with nonlinearity $V$ replaced by $\vtr$ will be denoted by $\utr$ and reads
\begin{equation}\label{int:truncated}
\begin{aligned}
\utr(x,t) \, = \; \int_0^1 g^h_t(x,y) u_0(y) \, dy  \; - \; \int_0^t \int_0^1 g^h_{t-s}(x,y) \vtr (u(y,s)) \, dy ds \\
+ \sqrt{2 \sigma } \int_0^t \int_0^1 g^h_{t-s}(x,y)  \, W( dy, ds)
\end{aligned}
\end{equation}

\begin{remark}
Note that \eqref{int:truncated} differs from \eqref{int:discreteintegralform} only by the truncation of the drift term. In particular, for all times $t\leq  \tau^h_Z$,
\[
u^h(x,t) = \utr (x,t) \qquad \forall x \in [0,1]
\]
where we defined the stopping time
\begin{equation}\label{int:tauzh}
 \tau^h_Z = \inf_{t \in [0,T]} \lbrace \|\utr \|_{\infty} >Z\rbrace \; = \; \inf_{t \in [0,T]} \lbrace \exists x : |\utr (x,t)| >Z\rbrace
\end{equation}
\hfill$\square$
\end{remark}
It is convenient to write 
\begin{equation}\label{int:splitting}
 \utr = v^h(x,t) - w^h_Z(x,t)
\end{equation}
where 
\begin{equation}\label{int:vh}
\begin{aligned}
v^h(x,t) \, = \; \int_0^1 g^h_t(x,y) u_0(y) \, dy 
\end{aligned}
\end{equation}
is the solution to the homogeneous problem, i.e. the discrete heat-type equation with long-range stencil, and 
\begin{equation}\label{int:whz}
\begin{aligned}
w^h_Z(x,t) \, &=  \; \int_0^t \int_0^1 g^h_{t-s}(x,y) \vtr (u(y,s)) \, dy ds \; +  \; \sqrt{2 \sigma } \int_0^t \int_0^1 g^h_{t-s}(x,y)  \, W( dy, ds) \\
& \overset{\eqref{int:Bh}}{=}  \int_0^t \int_0^1 g^h_{t-s}(x,y) \vtr (u(y,s)) \, dy ds \; +  \sqrt{2 \sigma } B^h(x,t)
\end{aligned}
\end{equation}

\begin{lemma}\label{lemma:truncatedbound}
 Let $u^h_0$ be the piecewise linear approximation of the initial data $u_0$. Let  $\utr$ be the solution to the system of SDEs with truncated nonlinearity $\vtr$ as defined in \eqref{int:truncated}.
Then, for all times $T$ and all $p>1$ there exists a constant $C(p, \gamma, T,Z)$ independently of $h$, such that 
\begin{equation}
 \sup_{[0,T] \times [0,1]} \E \left[ |\utr|^p \right] \leq C(\gamma, T, p, Z)
\end{equation}
\end{lemma}
\begin{proof}{Lemma}{lemma:truncatedbound}
Notice first that the solution to the heat equation is globally bounded: $\|v(t)\|_{L^2} \leq e^{-\lambda_{min} t} \|u_0\|_{L^2}$, where $\lambda_{min}$ is the smallest eigenvalue under periodic boundary  conditions.
Recall that $A^h_R$ is the positive definite coefficient matrix which contains $J_R(j)$ on the $j$-th subdiagonal. Therefore, its exponential has eigenvalues bounded by one, which leads to
\begin{equation}\label{int:vhbound}
\sup_{x \in [0,1]} |v^h(x,t)| \; = \sup_{1 \leq i \leq \frac{1}{h}} \Big|( e^{-\gamma A^h_R} u_0^h)_i\Big| \leq  \; \sup_{1 \leq i \leq \frac{1}{h}}|u_0(x_i)| \leq \sup_{x \in [0,1]} |u_0| \qquad \forall t > 0
\end{equation}
As $\vtr$ is bounded by M (see \eqref{int:VboundM}), we can estimate the moments of the second term, $w^h_Z(x,t)$,  for all $t <T$, using BDG and the estimate \eqref{int:dsg3} on the discrete semigroup:
\begin{equation}
\begin{aligned}
2^{-p}\E \left[ |w^h_Z(x,t)|^p\right] &\leq \;\E \left[ \left| \int_0^t \int_0^1 g^h_{t-s}(x,y) \vtr (u(y,s)) \, dy ds \right|^p\right] 
+ \sqrt{ 2 \sigma} \E \left[ \left|B^h(x,t) \right|^p\right]  \\
&\leq  M^p T^{p/2} \left| \int_0^t \int_0^1 g^h_{t-s}(x,y) \, dy ds \right|^{p/2} +  C_p \left|\int_0^t \int_0^1 g^h_{t-s}(x,y)^2  dy ds \right|^{p/2} \\
&\leq C(p,\gamma, T,Z)
\end{aligned}
\end{equation}
\end{proof}

\begin{lemma}\label{lemma:truncatedboundEsup}
 Let $u^h_0$ be the piecewise linear approximation of the initial data $u_0$. Let  $\utr$ be the solution to the system of SDEs with truncated drift \eqref{int:truncated}.
Then, for all times $T$ and all $p>1$ there exists a constant $C(p, \gamma, T,Z)$ independently of $h$, such that 
\begin{equation}\label{int:Esuputr}
\sup_h \E \left[ \sup_{[0,T] \times [0,1]} |\utr|^p \right] \leq C 
\end{equation}
\end{lemma}

\begin{proof}{Lemma}{lemma:truncatedboundEsup}
Thanks to \eqref{int:splitting}, we can write 
\begin{equation}\label{int:truncstart}
\begin{aligned}
\E \Big[ \sup_{[0,1] \times [0,T]} \left| \utr\right|^p \Big] \; \leq \; 2^p \E \left[ \sup_{[0,T] \times [0,1]} |v^h(x,t)|^p\right] + 2^p \E \left[ \sup_{[0,T] \times [0,1]} | w^h_Z(x,t)|^p \right]
\end{aligned}
\end{equation}
As seen in \eqref{int:vhbound}, for all $t>0$ we have
\begin{equation}\label{int:uzero}
 \E \left[ \sup_{[0,T] \times [0,1]} |v^h(x,t)|^p \right] \leq \E \left[ \sup_{x \in [0,1]} |u_0(x)|^p \right] \leq c
\end{equation}
Therefore, it remains to bound 
\begin{equation}
 \E \left[\sup_{[0,T] \times [0,1]} | w^h_Z(x,t)|^p \right] \;  \leq  \; 2^p( (I) + (II))
\end{equation}
where
\begin{equation}\label{int:vtrterm}
(I) =   \E \left[\sup_{[0,T] \times [0,1]} \left| \int_0^t \int_0^1 g^h_{t-s}(x,y) \vtr (u(y,s)) \, dy ds \right|^p \right] \overset{\eqref{int:dsg3}}{\leq} \; c(\gamma)  M^p  T^{p/4} 
\end{equation}
and $(II) = \sqrt{2 \sigma } \E \left[\sup_{[0,T] \times [0,1]} | B^h(x,t)|^p \right]$ with $B^h(x,t)$ as defined in \eqref{int:Bh}. We apply Lemma \ref{lemma:discSG6} to $(II)$ and get
\begin{equation}\label{int:Bhterm}
(II) =    \sqrt{2 \sigma }  \E \left[\sup_{[0,T] \times [0,1]} \left|  B^h(x,t)    \right|^p \right] \leq c(p,\gamma , t)
\end{equation}
We conclude from \eqref{int:uzero}, \eqref{int:vtrterm} and \eqref{int:Bhterm}
\begin{equation}
\begin{aligned}
\E \Big[ \sup_{[0,1] \times [0,T]} \left| \utr\right|^p \Big] \; \leq \; c(p,\gamma, T, Z)
\end{aligned}
\end{equation}
As the constant is independent of $h$, \eqref{int:Esuputr} follows.
\end{proof}

\subsection{Uniform moment bound without truncation}
We use the following comparison theorem to derive from Lemma \ref{lemma:truncatedbound} and Lemma \ref{lemma:truncatedboundEsup} the uniform moment bound.

\begin{proposition}\label{prop:comparison}
Let $u_1$ be the solution to
\begin{equation}
 \begin{aligned}
u^h_1(x,t) \, = \; \int_0^1 g^h_t(x,y) u^h_0(y) \, dy  \; - \; \int_0^t \int_0^1 g^h_{t-s}(x,y) V_1((u^h_1(y,s)) \, dy ds \\
+ \sqrt{2 \sigma } \int_0^t \int_0^1 g^h_{t-s}(x,y) \, W( dy, ds) \qquad \textup{ on } [0,1] \times \real^+
\end{aligned}
\end{equation}
with initial condition $u_1(x,0)$
and $u_2$ be the solution to 
\begin{equation}
 \begin{aligned}
u^h_2(x,t) \, = \; \int_0^1 g^h_t(x,y) u^h_0(y) \, dy  \; - \; \int_0^t \int_0^1 g^h_{t-s}(x,y) V_2((u^h_2(y,s)) \, dy ds \\
+ \sqrt{2 \sigma } \int_0^t \int_0^1 g^h_{t-s}(x,y) \, W( dy, ds)  \qquad \textup{ on } [0,1] \times \real^+
\end{aligned}
\end{equation}
with initial condition $u_2(x,0)$, and both equations are
subject to the same boundary conditions.
Suppose that one of the two verifies existence and uniqueness. 

If $V_1 \leq V_2$  holds and the initial conditions satisfy $u_1(x,0) \leq u_2(x,0)$, then, for all $t$ and $x$, 
\begin{equation}
 u^h_1(x,t) \leq u^h_2(x,t) \qquad \textup{ almost surely }
\end{equation}
\end{proposition}
This result is taken from \cite{geissmanthey}. \hfill $\square$

Define now the one-sided truncations 
\begin{equation}
\begin{aligned}
V^+_Z(u) &= V'(u)\indicator_{( -\infty , Z]} + V'(Z)\indicator_{[Z, \infty)}\\
V^-_Z(u) &= V'(u)\indicator_{[-Z , \infty ) } + V'(-Z)\indicator_{( -\infty , -Z]}
\end{aligned}
\end{equation}
and denote by $(\utr)^+$ and $(\utr)^-$ the mild solutions to the associated truncated problems. 
Note that for fixed $Z$,
\begin{equation}
V^-_Z(u)  \leq  - \vtr (u) \leq - V^+_Z(u)
\end{equation}
such as
\begin{equation}
V^-_Z(u)  \leq  -  V' (u) \leq - V^+_Z(u)
\end{equation}
We see that the same moment bounds hold on $\utrp$ and $\utrm$  via comparison:
\begin{lemma}\label{lemma:maximal}
 Let $u^h_0$ be the piecewise linear approximation of the initial data $u_0$. Let  $\utrp$ and $\utrm$ the mild solutions to the system of SDEs with truncated drift $V^+_Z(u)$ and $V^-_Z(u)$, respectively.
Then, for all times $T$ and all $p>1$ there exists a constant $C(p, \gamma, T,Z)$ independently of $h$, such that 
\begin{equation}
\sup_h \E \left[ \sup_{[0,T] \times [0,1]} |\utrp|^p \right] \leq C(p, \gamma, T,Z)
\end{equation}
and \begin{equation}
\sup_h \E \left[ \sup_{[0,T] \times [0,1]} |\utrm|^p \right] \leq C(p, \gamma, T,Z)
\end{equation}
\end{lemma}
\hfill $\square$

\begin{proposition}\label{prop:moments}
 Let $u^h_0$ be the piecewise linear approximation of the initial data $u_0$. Let  $u^h$ be the solutions to the system of SDEs \eqref{int:discreteintegralform}.
Then, for all times $T>0$ and all $p>1$ there exists a constant $C(p, \gamma, T)$ independently of $h$, such that 
\begin{equation}
\sup_h \E \left[ \sup_{[0,T] \times [0,1]} |u^h(x,t)|^p \right] \leq C 
\end{equation}
\end{proposition}
\begin{proof}{Proposition}{prop:moments}
We apply the comparison theorem \ref{prop:comparison} for 
\begin{equation}
V^-_Z(u)  \leq  -  V' (u) \leq - V^+_Z(u)
\end{equation}
to get for all $(x,t)$
\begin{equation}
\utrm (x,t) \; \leq \;  \utr (x,t) \; \leq \; \utrp(x,t)
\end{equation}
where $\utrp$ and $\utrm$ satisfy truncated moment bounds as proved above. Note that 
\begin{equation}
|u^h(x,t)|  \leq \indicator_{u^h < 0} |\utrm(x,t)| + \indicator_{u^h > 0} |\utrp(x,t)|
\end{equation}
Therefore,
\begin{equation}
\sup_{[0,T] \times [0,1]} |u^h(x,t)|^p \leq \; 2^p \left( \sup_{[0,T] \times [0,1]}  |\utrm(x,t)|^p + \sup_{[0,T] \times [0,1]} |\utrp(x,t)|^p \right)
\end{equation}
which implies in particular
\begin{equation}
\sup_h \E \left[ \sup_{[0,T] \times [0,1]} |u^h(x,t)|^p \right] \leq C 
\end{equation}
\end{proof}

\section{Convergence of solutions}
As above in the truncated case, we write $u^h(x,t) = v^h(x,t) - w^h(x,t)$ where $v^h(x,t)$ as in \eqref{int:vh} and 
 $w^h(x,t) = u^h(x,t) - v^h(x,t)$, i.e.
\begin{equation}
\begin{aligned}
w^h(x,t) \, =   \; -\; \int_0^t \int_0^1 g^h_{t-s}(x,y) V'(u(y,s)) \, dy ds + \sqrt{2 \sigma } \int_0^t \int_0^1 g^h_{t-s}(x,y)  \, W( dy, ds)
\end{aligned}
\end{equation}

\subsection{Convergence of the homogeneous solutions}
\begin{proposition}\label{prop:homogeneous}
 Let the initial data $u_0$ be in $C^4$. Let $v^h(x,t)$ be the solution to the discrete homogeneous equation as defined in \eqref{int:vh}.
Then we have for all $t_0 >0$ and $\zeta < \frac{1}{2}$
\begin{equation}\label{int:c4heatconv}
\sup_{[t_0,T] \times [0,1]} | v^h(x,t) - v(x,t)|  \leq c(\gamma, T) h^{2 - 2\zeta}
\end{equation}
\end{proposition}

\begin{proof}{Proposition}{prop:homogeneous}
We know by semigroup properties that $\partial_t g_t u_0 = g_t \Delta u_0$ and  $\partial_t g_t^h u_0 = g_t^h A^h_R u_0$, therefore
 \begin{equation}
\begin{aligned}
v(x,t) \, & = \; \int_0^1 g_t(x,y) u_0(y) \, dy \; = \;  u_0(x) \; + \;  \int_0^t \int_0^1 g_s(x,y) \Delta u_0(y) \, dy ds \\
v^h(x,t) \, & = \; \int_0^1 g^h_t(x,y) u^h_0(y) \, dy \; = \;  u^h_0(x) \; + \;  \int_0^t \int_0^1 g^h_s(x,y) A^h_R u_0(y) \, dy ds \\
\end{aligned}
\end{equation}
This means that 
 \begin{equation}\label{int:homconvc4}
\begin{aligned}
\sup_{[0,T] \times [0,1]} | v^h(x,t) - v(x,t)| &\leq \sup_{x \in [0,1]} |u^h_0(x) - u_0(x)| \\
		     & + \sup_{[0,T] \times [0,1]} \left|  \int_0^t \int_0^1 g^h_s(x,y)\big( A^h_R u_0(y) - \Delta u_0(y)\big) \, dy ds  \right| \\
		     & + \sup_{[0,T] \times [0,1]} \left|  \int_0^t \int_0^1 \big(g^h_s(x,y) - g_s (x,y)  \big) \Delta u_0(y) \, dy ds  \right| \\
\end{aligned}
\end{equation}
The first term in \eqref{int:homconvc4} can be estimated as
\begin{equation}\label{int:icconv}
\sup_{x \in [0,1]} |u^h_0(x) - u_0(x)| \leq c \cdot  h^2.
\end{equation}
By \eqref{int:consistency}, $A^h_R$ is a stencil on a uniform grid with consistency order 2.  Using \eqref{int:dsg3}, we get
 \begin{equation}\label{int:hompart2}
\begin{aligned}
\sup_{[0,T] \times [0,1]} \left|  \int_0^t \int_0^1 g^h_s(x,y)\big( A^h u_0(y) - \Delta u_0(y)\big) \, dy ds  \right| \leq c(\gamma, T) h^2
\end{aligned}
\end{equation}
For the last term, we evoke Proposition \ref{prop:sgconv} on the convergence of semigroups
 \begin{equation}\label{int:hompart3}
\begin{aligned}
 & \sup_{[0,T] \times [0,1]} \left|  \int_0^t \int_0^1 \big(g^h_s(x,y) - g_s (x,y)  \big) \Delta u_0(y) \, dy ds  \right| \\
  &\leq \;  \left\| \Delta u_0 \right\|_{L^2([0,T] \times [0,1])} \cdot  \sup_{ x \in [0,1]} \left( \int_0^{\infty} \int_0^1 \big(g^h_s(x,y) - g_s (x,y)  \big)^2 \, dy ds\right)^{1/2} \\
  &\leq c(\gamma, T )\cdot h^{2 - 2\zeta}
\end{aligned}
\end{equation}
Consequently, \eqref{int:icconv}, \eqref{int:hompart2} and \eqref{int:hompart3} give
 \begin{equation}
\begin{aligned}
\sup_{[0,T] \times [0,1]} | v^h(x,t) - v(x,t)| &\leq  c(\gamma, T) \left( h^{2 - 2\zeta} +  h^2 \right)
\end{aligned}
\end{equation}
which proves \eqref{int:c4heatconv}.
\end{proof}

\subsection{Convergence of truncated solutions}
Recal the splitting \eqref{int:splitting} where $ \utr = v^h(x,t) + w^h_Z(x,t)$.
Similarly, we split the solution to the continuous truncated equation \eqref{int:uz} as 
\begin{equation}
 u_Z(x,t) = v(x,t) + w_Z(x,t)
\end{equation}
where $v(x,t)$ is the solution to the heat equation and $w_Z$ the nonlinear term and the stochastic integral. 
After having showed convergence of $v^h(x,t)$ to  $v(x,t)$ in Proposition \ref{prop:homogeneous}, we now study the convergence of the  truncated nonlinear term and the stochastic convolution.

\begin{proposition}\label{prop:wtruncatedconv}
  Let $u^h_0$ be the piecewise linear approximation of the initial data $u_0$. Let  $\utr$ be the solution to the system of SDEs with truncated drift \eqref{int:truncated}, $v^h(x,t)$ be as in \eqref{int:vh}  and let $ w^h_Z(x,t) =  \utr - v^h(x,t)$
Then, for all times $T >0$, $\zeta < \frac{1}{2}$ and all $p>1$ there exists a constant $C(p, \delta, \gamma, T,Z)$ independently of $h$, and an exponent $\delta = \delta (p) >0$ such that
\begin{equation}
 \E \left[ \sup_{[0,T] \times [0,1]} |w^h_Z - w_Z|^p \right]^{1/p} \leq  C h^{\frac{1}{2} - \delta}
\end{equation}
\end{proposition}
\begin{proof}{Proposition}{prop:wtruncatedconv}
By definition,
\begin{equation}
\begin{aligned}
w^h_Z(x,t) - w_Z(x,t) \, &=  \; \int_0^t \int_0^1 g^h_{t-s}(x,y) \vtr (\utr(y,s)) -   g_{t-s}(x,y) \vtr (u_Z(y,s)) \, dy ds \\
 &+ \sqrt{2 \sigma } \int_0^t \int_0^1 g^h_{t-s}(x,y) - g_{t-s}(x,y) \, W( dy, ds) 
\end{aligned}
\end{equation}
We have 
\begin{equation}\label{int:IuII}
\begin{aligned}
|w^h_Z(x,t) - w_Z(x,t)|^p \, \leq   \;  2^p \Big( (I) + (II)\Big) 
\end{aligned}
\end{equation}
with 
\begin{equation}
\begin{aligned}
(I) &:=  \left|\int_0^t \int_0^1 g^h_{t-s}(x,y) \vtr (u^h(y,s)) -   g_{t-s}(x,y) \vtr (u_Z(y,s)) \, dy ds\right|^p \\
(II)&:= (2 \sigma)^{p/2} \left|\int_0^t \int_0^1 g^h_{t-s}(x,y) - g_{t-s}(x,y) \, W( dy, ds)\right|^p
\end{aligned}
\end{equation}
We first estimate term (I), which we again split in two parts: 

\underline{Step 1: Estimates on term (I)}
Note that by \eqref{int:VboundM}
\begin{equation}\label{int:VboundM2}
 \left|\int_0^t \int_0^1  |\vtr (u_Z(y,s))|^2 \, dy ds\right|^{p/2} \leq  \left|\int_0^t  M^2  ds\right|^{p/2} = M^p t^{p/2}
\end{equation}
We split
\begin{equation}
\begin{aligned}
 (I) &\leq 2^p \Big( \left|\int_0^t \int_0^1 \left( g^h_{t-s}(x,y) -   g_{t-s}(x,y)\right) \vtr (u_Z(y,s)) \, dy ds\right|^p \\
  &+  \; \left|\int_0^t \int_0^1 g^h_{t-s}(x,y) \left( \vtr (\utr(y,s)) -  \vtr (u_Z(y,s)) \right) \, dy ds\right|^p \Big) = 2^p((Ia) + (Ib)) \\
\end{aligned}
\end{equation}
For $(Ia)$, we use Cauchy-Schwarz and employ Proposition \ref{prop:sgconv}
\begin{equation}\label{int:gvconv1}
\begin{aligned}
(Ia) &= \left|\int_0^t \int_0^1 \left( g^h_{t-s}(x,y) -   g_{t-s}(x,y)\right) \vtr (u_Z(y,s)) \, dy ds\right|^p \\
&\leq \left(\int_0^t \int_0^1 \left| g^h_{t-s}(x,y) -   g_{t-s}(x,y)\right|^2 \, dy ds\right)^{p/2} \; 
\left(\int_0^t \int_0^1  |\vtr (u_Z(y,s))|^2 \, dy ds\right)^{p/2} \\
&\overset{\eqref{int:VboundM2}}{\leq} C(p, \gamma, t,  lip(\vtr) )  h^{2p( 1 - \zeta)}
\end{aligned}
\end{equation}
For $(Ib)$ we estimate, using lipschitz continuity of $\vtr$,
\begin{equation}\label{int:gvconv2}
\begin{aligned}
(Ib) &= \left|\int_0^t \int_0^1 g^h_{t-s}(x,y) \Big( \vtr (\utr(y,s)) -  \vtr (u_Z(y,s)) \Big) \, dy ds\right|^p \\
&\leq \left|\int_0^t \int_0^1 (g^h_{t-s}(x,y))^2 \, dy ds\right|^{p/2} \; \left|\int_0^t \int_0^1 \left( \vtr (\utr(y,s)) \; -  \vtr (u_Z(y,s)) \right)^2 \, dy ds\right|^{p/2} \\
&\overset{\eqref{int:dsg3}}{\leq} C(p,\gamma, t, lip(\vtr)) \;  \left(\int_0^t \int_0^1 \left| \utr (y,s) - u_Z(y,s) \right|^2 \, dy ds\right)^{p/2} \\
\end{aligned}
\end{equation}
Using that $\utr$ solves \eqref{int:truncated} and $u_Z$ solves \eqref{int:uz}, we can estimate the RHS in \eqref{int:gvconv2} further as
\begin{equation}\label{int:gvconv3}
\begin{aligned}
\left|\int_0^t \int_0^1 \left| \utr (y,s) - u_Z(y,s) \right|^2 \, dy ds\right|^{p/2} 
\leq t^{\frac{p-2}{p}} \int_0^t \int_0^1 \left| \utr (y,s) - u_Z(y,s) \right|^p \, dy ds \\
\leq 2^p T^{\frac{p-2}{p}} \Big( t \sup_{[0,1] \times [0,T]} |v(y,s) - v^h(y,s)|^p + \, \int_0^t   \sup_{[0,1] \times [0,\tilde{s}]}  \left| w_Z(y,s) - w^h_Z(y,s) \right|^p \, d\tilde{s} \Big) \\
\end{aligned}
\end{equation}
Summing up estimates \eqref{int:gvconv1} - \eqref{int:gvconv3}, we get the following estimate for (I): 
\begin{equation}
\begin{aligned}
 (I) &\leq  2^p \Big( (Ia) + (Ib) \Big) \\
  &\lesssim   h^{2p( 1 - \zeta)} \; + \; \sup_{[0,1] \times [0,T]} |v(y,s) - v^h(y,s)|^p 
   + \,  \int_0^t   \sup_{[0,1] \times [0,\tilde{s}]}  \left| w_Z(y,s) - w^h_Z(y,s) \right|^p \, d\tilde{s} \\
\end{aligned}
\end{equation}
where we used the notation $\lesssim$ to avoid writing out the constant $ C(p, \gamma, T, lip(\vtr))$ on the RHS.
As the lipschitz constant of $\vtr$ depends only on $Z$, we can write $ C(p, \gamma, T, lip(\vtr)) =  C(p, \gamma, T, Z)$.

\underline{Step 2: Estimates on term (II)}
We have 
\begin{equation}
\begin{aligned}
(II)&= (2 \sigma)^{p/2} \left|\int_0^t \int_0^1 g^h_{t-s}(x,y) - g_{t-s}(x,y) \, W( dy, ds)\right|^p\\
&=   (2 \sigma)^{p/2} \left| B^h(x,t) - B(x,t)\right|^p\\
\end{aligned}
\end{equation}
Thanks to Lemma \ref{lemma:discSG6} and its continuous counterpart, we can estimate
\begin{equation}
\begin{aligned}
\left| B^h(x,t) - B(x,t) - B^h(x',t) + B(x',t)\right| \leq Y^h |x - x'|^{\frac{1}{2} - \delta}
\end{aligned}
\end{equation}
Therefore, with $\theta = \frac{1}{2} - \delta$
\begin{equation}
\begin{aligned}
\sup_{x\in [0,1]} \left| B^h(x,t) - B(x,t)\right| \; \leq \max_{1 \leq i \leq 1/h} | B^h_i(t) - B(x_i , t)| + Y^h h^{\theta}
\end{aligned}
\end{equation} 
This is independent of $t$, so we get
\begin{equation}\label{int:bh}
\begin{aligned}
\sup_{[0,1]\times [0,T]} \left| B^h(x,t) - B(x,t)\right| \; \leq \max_{1\leq i \leq 1/h} \sup_{t\in [0,T]}  | B^h_i( t) - B(x_i , t)| + Y^h h^{\theta}
\end{aligned}
\end{equation} 
Recall that $(II) =   (2 \sigma)^{p/2} \left| B^h(x,t) - B(x,t)\right|^p$, so we take \eqref{int:bh} to the $p$-th power to get
\begin{equation}
\begin{aligned}
\E \left[ \sup_{[0,1]\times [0,T]} (II) \right] 
&\leq \; 2^p \E \left[\max_{i = 1 \ldots 1/h} \sup_{t \in [0,T]}  | B^h_i (t) - B(x_i , t)|^p\right]   + (2\sigma)^{p/2} \E \left[   Y^p_h \right] h^{\theta p}  \\
&\overset{\eqref{int:yhass}}{\leq } \; 2^p \sup_{x \in [0,1]} \underbrace{\E \left[ \sup_{t \in [0,T]}  | B^h(x_i, t) - B(x_i , t)|^p \right]}_{\ast}  +  c(p,T, \delta , \sigma) h^{\theta p}  \\
  &\overset{BDG}{\leq } c(p) \left| \int_0^T \int_0^1 \left(g^h_{t-s}(x_i,y) - g_{t-s}(x_i,y)\right)^2 \, dy ds \right|^{\frac{p}{2}} + c(p,T, \delta , \sigma) h^{\theta p}  \\ 
  &\overset{\textup{Prop.}\ref{prop:sgconv} }{\leq} c(p , \gamma , \delta , \sigma, T ) \left( h^{2p(1-\zeta)}  +  h^{\frac{p}{2} - \delta p } \right) \\
\end{aligned}
\end{equation}
\underline{Step 3: Conclusion}
To sum up, we know with \eqref{int:IuII} that $\E \left[ \sup_{[0,1]\times [0,T]} |w^h_Z(x,t) - w_Z(x,t)|^p \right] \leq \E \left[ \sup_{[0,1]\times [0,T]} (I) + (II)  \right]$ and
\begin{equation}
\begin{aligned}
\E \left[ \sup_{[0,1]\times [0,T]} (I) + (II)  \right] 
&\lesssim \left( h^{2p(1-\zeta)}  +  h^{\frac{p}{2} - \delta p} \right) 
+ \E \left[ \sup_{[0,1]\times [0,T]} |v(y,s) - v^h(y,s)|^p  \right]\\
 &+ \,  \E \left[ \int_0^t   \sup_{[0,1] \times [0,\tilde{s}]}  \left| w_Z(y,s) - w^h_Z(y,s) \right|^p \, d\tilde{s}  \right]
 \end{aligned}
\end{equation}
where we used the notation $\lesssim$ to avoid writing out the constant $ c(p , \gamma , \delta , t)$ on the RHS.
Gronwall's Lemma with $f(t) := \E \left[ \sup_{[0,1]\times [0,t]} |w^h_Z(x,t) - w_Z(x,t)|^p \right] $ leads 
\begin{equation}
\begin{aligned}
\E \left[ \sup_{[0,1]\times [0,T]} |w^h_Z(x,t) - w_Z(x,t)|^p \right] 
&\leq   \;  c(p , \gamma , \delta , t) \left( h^{2p(1-\zeta)}  +  h^{\frac{p}{2} - \delta p } \right) \\
&+ \E \left[ \sup_{[0,1]\times [0,T]} |v(y,s) - v^h(y,s)|^p  \right]\\
 \end{aligned}
\end{equation}
Proposition \ref{prop:homogeneous} gives 
\begin{equation}\label{int:c4heatconv2}
\E \left[ \sup_{[0,T] \times [0,1]} | v^h(x,t) - v(x,t)|^p  \right]  \leq c(\gamma, t) h^{p(2 - 2\zeta)}
\end{equation}
which gives the final estimate
\begin{equation}
\begin{aligned}
\E \left[ \sup_{[0,1]\times [0,T]} |w^h_Z(x,t) - w_Z(x,t)|^p \right]^{1/p}
&\leq   \; c(p , \gamma , \delta , t) \left( h^{\frac{1}{2} - \delta } + h^{2-2\zeta} \right) \\
 \end{aligned}
\end{equation}
Taking the $p$-th root leads the desired result.
\end{proof}

From the above propositions, we can conclude convergence of truncated solutions in $L^p(\Omega, C([0,1] \times [0,T]))$ and almost surely.
\begin{proposition}\label{prop:truncatedconv}
  Let $u^h_0$ be the piecewise linear approximation of an initial data $u_0 \in C^{4}$. Let  $\utr$ be the solution to the system of SDEs with truncated drift \eqref{int:truncated} and $u_Z$ the solution to the continuous truncated equation \eqref{int:uz}.
Then, for all times $T>0$ and all $p>1$ there exists a constant $C(p, \gamma, T,Z)$ independently of $h$, such that for $\zeta < \frac{1}{2}$
\begin{equation}\label{int:truncatedconv}
 \E \left[ \sup_{[0,T] \times [0,1]} |\utr - u_Z|^p \right]^{1/p} \leq  C  h^{\frac{1}{2}-\delta}  
\end{equation}
\end{proposition}
\begin{proof}{Proposition}{prop:truncatedconv}
Proposition \ref{prop:homogeneous} and \ref{prop:wtruncatedconv} give for $T>0$ 
\begin{equation}
 \E \left[ \sup_{[0,T] \times [0,1]} |\utr - u_Z|^p \right]^{1/p} \leq  C(p, \gamma, T,Z) \left( h^{2(1-\zeta)}  +  h^{\theta} \right)
\end{equation}
As for $\zeta < \frac{1}{2},  \min \lbrace \frac{1}{2}-\delta , 2 - 2 \zeta \rbrace = \frac{1}{2}-\delta$, the result follows.
\end{proof}

\begin{proposition}\label{prop:uniformtruncated}
  Let $u^h_0$ be the piecewise linear approximation of an initial data $u_0 \in C^{4}$. Let  $\utr$ be the solution to the system of SDEs with truncated drift \eqref{int:truncated} and  $u_Z$ be the solution to the continuous truncated equation \eqref{int:uz}.
Then, for all times $T > 0$ and all $\frac{2}{\sqrt{3}} < Z < \infty$ there exists a random variable $\XX_Z$, which is almost surely finite, such that for $\zeta < \frac{1}{2}$ and $\eta < \frac{1}{2}-\delta $

\begin{equation}\label{int:uniformtruncated}
 \sup_{[0,T] \times [0,1]} |\utr - u_Z| \;  \leq \;  \XX_Z  h^{\eta}  
\end{equation}
In particular, $\utr$ converges uniformly to $u_Z$ almost surely.
\end{proposition}

\begin{proof}{Proposition}{prop:uniformtruncated} 
Define the random variable
\begin{equation}
  \XX_Z^h  :=  h^{-\eta} \sup_{[0,T] \times [0,1]} |\utr - u_Z|
\end{equation}
Note with the Markov inequality that
\begin{equation}
\P [(\XX_Z^h)^p \geq 1 ]  \leq \E [(\XX_Z^h)^p] 
\end{equation}
As with  $X :=  \sup_{[0,T] \times [0,1]} |\utr - u_Z| $ 
$ \E [(\XX_Z^h)^p] = \E [ |X \cdot h^{-\eta} |^p ] = h^{-\eta p} \E [|X|^p] $, we can use Proposition \ref{prop:truncatedconv} to show finiteness of moments of $\XX_Z^h$ as long as $\eta <  \frac{1}{2}-\delta$. 
For such $\eta$, Borel-Cantelli can be applied and gives
\begin{equation}x
\P \left[\left( \limsup_{h \to 0} \lbrace (\XX_Z^h)^p \geq 1 \rbrace \right)^c \right]  \leq \E [(\XX_Z^h)^p] \P \left[ \bigcup_{N\geq 1} \bigcap_{M \geq N} \lbrace (\XX_Z^h)^p \leq 1 \rbrace  \right]  = 1
\end{equation}
Therefore, there exists an $N=N(\omega)$ such that for all $M \geq N$, $ (\XX_Z^h)^p(\omega ) \leq 1 $, and consequently
\begin{equation}\label{int:suphxiZ}
\sup_h (\XX_Z^h)^p (\omega ) = \max_{1 \leq k \leq N(\omega )} (\XX_Z^h)^p   + \sup_{k \geq N(\omega)} (\XX_Z^h)^p \leq \max_{1 \leq k \leq N(\omega )} (\XX_Z^h)^p  + 1 < \infty
\end{equation}
Therefore $\XX_Z := \sup_h \XX_Z^h $ is a.s. finite.
\end{proof}

\subsection{From truncated to non-truncated solutions}

Recall the discrete stopping time  $\tau^h_Z$ introduced in \eqref{int:tauzh}
\[
 \tau^h_Z = \inf_t \{ \|\utr (t)\|_{\infty} >Z \} = \inf_t \{ \exists x : |\utr (x,t)| > Z \}
\]
we define furthermore
\[
 \tau_Z = \inf_t \{ \|u_Z (t)\|_{\infty} >Z \} = \inf_t \{ \exists x : |u_Z(x,t)| > Z \}
\]
and
\begin{equation}\label{int:defomegaz}
  \Omega_Z = \left\{ \tau_{Z - \delta} > T \textup{ and } \liminf_{h\to 0} \tau^h_Z > T \right\}
\end{equation}
$\Omega_Z$  can be decomposed as
\begin{equation}
  \Omega_Z = \bigcup_{h_0}  \Omega_{Z, h_0}
\end{equation}
where
\begin{equation}
  \Omega_{Z, h_0} = \bigcap_{h \leq h_0} \left\{ \tau_{Z - \delta} > T \textup{ and }  \tau^h_Z > T \right\}
\end{equation}
Note that $\Omega_Z \subset \Omega_{Z+1}$ and that the following lemma holds:

\begin{lemma}\label{lemma:goodsets}
For the sets $\Omega_Z$ defined in \eqref{int:defomegaz} holds, under the conditions of Proposition \ref{prop:uniformtruncated},
\[
 \P[\Omega_Z] \longrightarrow 1 \textup{ for } Z \to \infty
\]
\end{lemma}
\begin{proof}{Lemma}{lemma:goodsets}
We prove the statement on the complement
\begin{equation}\label{int:badsetzero}
\P[\Omega_Z^c] \longrightarrow 0 \textup{ for } Z \to \infty
\end{equation}
\underline{Step 1}:As a first step, we show that 
\begin{equation}\label{int:liminftauh}
 \liminf_{h\to 0} \tau_Z^h \geq T
\end{equation}
Recall that under the conditions of Proposition \ref{prop:uniformtruncated}, 
$  \sup_{[0,T] \times [0,1]} |\utr - u_Z| \;  \leq \;  \XX_Z^h  h^{\eta} $
and so for all $h \leq h(M) = \left(\frac{M}{\delta}\right)^{\eta}$ we get $  \sup_{[0,T] \times [0,1]} |\utr - u_Z| \;  \leq \;  \delta $ and therefore
\begin{equation}
   \sup_{[0,T] \times [0,1]} |\utr | =   \sup_{[0,T] \times [0,1]} | \utr - u_Z + u_Z | \leq \sup_{[0,T] \times [0,1]} |u_Z| +  \delta  
\end{equation}  
which gives the equality of the events
\begin{equation}
\lbrace \tau_{Z - \delta} \geq T \rbrace \; = \;  \lbrace \sup_{[0,T] \times [0,1]} |u_Z|  < Z - \delta \rbrace \; = \;  \lbrace  \sup_{[0,T] \times [0,1]} |\utr | < Z \rbrace
\end{equation}
We conclude that for any $\omega \in \lbrace \tau_{Z - \delta} \geq T \rbrace$ we have  $\sup_{[0,T] \times [0,1]} |\utr | < Z$, which is nothing else than 
\begin{equation}
 \tau_Z^h > T \qquad \textup{for all } h \leq h(M) = \left(\frac{M}{\delta}\right)^{\eta}
\end{equation}
which proves \eqref{int:liminftauh}.

\noindent \underline{Step 2: Proof of equation \eqref{int:badsetzero}}
 We conclude from \eqref{int:liminftauh} that for all $M>0$
\begin{equation}\label{int:liminfsets}
 \lbrace \liminf_{h \to 0} \tau_{Z - \delta}^h < T \rbrace  \cap  \lbrace  \XX_Z < M \rbrace   \subset \lbrace \tau_{Z - \delta} < T \rbrace  \cap  \lbrace \XX_Z < M \rbrace 
\end{equation}
We have the following decomposition of $\P[\Omega_Z^c] $ for $0 < \delta < 1$ fixed, for all $Z > \frac{2}{\sqrt{3}}$ and  $M >0$:
\begin{equation}\label{int:pomegazcomp}
\P[\Omega_Z^c] = \P[\tau_{Z-\delta} \leq T] + \P \left[\liminf_{h\to 0} \tau_Z^h < T; \XX_Z < M \right] + \P [ \XX_Z > M]
\end{equation}

We estimate the terms arising in \eqref{int:pomegazcomp}:
As $\XX_Z$ is almost surely finite by \eqref{int:suphxiZ} from the last lemma, $ \P [ \XX_Z > M] \to 0$ for  $M \to \infty$, so the last term in \eqref{int:pomegazcomp} vanishes in the large $M$-limit. Moreover, by definition of the stopping time, 
\begin{equation}
 \lbrace \tau_{Z - \delta} < T \rbrace \subset \lbrace \sup_{ [0,T] \times [0,1]} |u|  > Z - \delta \rbrace 
\end{equation}
and by Markov
\begin{equation}
\P \left[ \lbrace \sup_{ [0,T] \times [0,1]} |u|  > Z - \delta \rbrace \right] \leq \frac{1}{(Z - \delta)^p} \E \left[ \sup_{ [0,T] \times [0,1]} |u|^p \right] \overset{\eqref{int:continousmoments}}{\leq}   \frac{C(T,p)}{(Z - \delta)^p}
\end{equation}
as $Z \to \infty$ as $\E \left[ \sup_{ [0,T] \times [0,1]} |u|^p \right] < C(T,p) $ due to Proposition \ref{prop:continuoussolution}. 
With \eqref{int:liminfsets}  we can bound for all $M>0$
\begin{equation}
 \P \left[\liminf_{h\to 0} \tau_Z^h < T; \XX_Z < M \right] \leq  \P \left[\tau_{Z-\delta} < T; \XX_Z < M \right] \leq \P[\tau_{Z-\delta} < T] 
\end{equation}
which brings us to the same case as before. We summarize 
\begin{equation}\label{int:pomegazcomp2}
\begin{aligned}
\P[\Omega_Z^c] &= \P[\tau_{Z-\delta} \leq T] + \P \left[\liminf_{h\to 0} \tau_Z^h < T; \XX_Z < M \right] + \P [ \XX_Z > M] \\
&\leq 2  \frac{C(T,p)}{(Z - \delta)^p} + \P [ \XX_Z > M]
\end{aligned}
\end{equation}
and conclude \eqref{int:badsetzero} by first taking the limit in $M$ and then in $Z$.
\end{proof}

In fact, \eqref{int:badsetzero} can be quantifies as such: For any given $\epsilon >0$  we can choose $Z$ such that 
\begin{equation}
\P[\Omega_Z^c] \leq \epsilon
\end{equation}
As the sets $\Omega_{Z, h_0}$ are ordered and increase when $h_0$ is decreases, 
\begin{equation}\label{int:OmegaZ}
\P[\Omega_Z^c] = \P \left[ \bigcap_{h_0} \Omega^c_{Z,h_0}  \right] = \lim_{h_0 \to 0} \P \left[ \Omega^c_{Z,h_0} \right] \leq \epsilon
\end{equation}

We have already proved almost sure convergence for truncated solutions in \eqref{int:uniformtruncated}, which gave a r.v. $\XX_Z$. Now have to remove the truncation.

\begin{theorem}\label{theo:asconvergence}
  Let $u^h_0$ be the piecewise linear approximation of an initial data $u_0 \in C^{4}$. Let  $u^h(x,t)$ be the solution to the system of SDEs \eqref{int:discretesystem} and $u(x,t)$ the solution to \eqref{int:generalac}.
  
Then, for all times $T > 0$, $\zeta < \frac{1}{2}$ and $\eta < \frac{1}{2}-\delta $ there exists a random variable $\XX$, which is almost surely finite,
\begin{equation}\label{int:asconvergence}
 \sup_{[0,T] \times [0,1]} |u^h(x,t) - u(x,t) | \;  \leq \;  \XX  h^{\eta}  
\end{equation}
In particular, we have almost surely uniform convergence of $u^h$ to $u$.
\end{theorem}

\begin{proof}{Theorem}{theo:asconvergence}
Recall from above the definition
\[
   \Omega_Z = \bigcup_{h_0} \Omega_{Z, h_0} = \bigcup_{h_0} \left( \bigcap_{h \leq h_0} \left\{ \tau_{Z - \delta} > T \textup{ and }  \tau^h_Z > T \right\} \right)
\]
We know that the sets $\Omega_Z$ defined in \eqref{int:defomegaz} are increasing in the sense of $\Omega_Z \subset \Omega_{Z+1}$ and $ \P[\Omega_Z] \longrightarrow 1 $ as $Z \to \infty$ by Lemma \ref{lemma:goodsets}. We conclude that 
\begin{equation}
 \P \left[ \bigcup_{Z\geq 1} \Omega_Z \right] = 1
\end{equation}
Define now
\begin{equation}
\Omega_{Z, \XX_Z} := \Omega_Z \cap \lbrace \XX_Z < \infty \rbrace
\end{equation}
As $\XX_Z := \sup_h \XX_Z^h$ is a.s. finite, the union of $\Omega_{Z, \XX_Z}$ has again full measure:
\begin{equation}\label{int:goodset1}
 \P \left[ \bigcup_{Z\geq 1} \Omega_{Z, \XX_Z}\right] = 1
\end{equation}
On $\Omega_{Z, \XX_Z}$ there exists $h_0(\omega)$ such that for all $h \leq h_0(\omega)$ we are in the good regime $\tau_Z^h >T$ and $\tau_{Z-\delta} >T$.
In this regime, we have according to Proposition \ref{prop:uniformtruncated}
\begin{equation}
  \sup_{[0,T] \times [0,1]} |u^h(x,t) - u(x,t) | \;  = \; \sup_{[0,T] \times [0,1]} |\utr - u_Z| \;  \leq \;  \XX_Z  h^{\eta} 
\end{equation}
as $\XX_Z := \sup_h \XX_Z^h$ is finite on $\Omega_{Z, \XX_Z}$.
Therefore, for sufficiently small $h$ there exists a finite random variable $ \widetilde{\XX_Z}$ such that 
\begin{equation}\label{int:XztoX}
  \sup_{[0,T] \times [0,1]} |u^h(x,t) - u(x,t) |   \leq \;  \widetilde{\XX_Z} h^{\eta} 
\end{equation}
By choosing 
\begin{equation}
 \XX := \widetilde{\XX_Z} \textup{ on }  \Omega_{Z, \XX_Z} \setminus \Omega_{Z-1, \XX_{Z-1}}
\end{equation}
we nest the inequalities \eqref{int:XztoX} to construct an almost surely finite random variable $\XX $ on $\bigcup_{Z>1} \Omega_{Z, \XX_Z}$. As this union has full measure due to \eqref{int:goodset1}, we conclude
\begin{equation}
  \sup_{[0,T] \times [0,1]} |u^h(x,t) - u(x,t) |   \leq \;  \XX h^{\eta} 
\end{equation}
which ends the proof.
\end{proof}

\begin{theorem}\label{theo:lpconvergence}
Let $u^h_0$ be the piecewise linear approximation of an initial data $u_0 \in C^{4}$. Let  $u^h(x,t)$ be the solution to the system of SDEs \eqref{int:discretesystem}  and $u(x,t)$ the solution to \eqref{int:generalac}.
  
Then, for all times $T>0$, $\delta >0$, $p>1$ and  $\zeta < \frac{1}{2}$
   \[
   \E \left[ \sup_{[0,T] \times [0,1]} |u^h - u|^p\right]^{1/p}  \leq c h^{\frac{1}{2} - \delta}  
  \]
\end{theorem}

\begin{proof}{Theorem}{theo:lpconvergence}
We split the expectation into two parts
\begin{equation}\label{int:step1}
\begin{aligned}
  \E \left[  \sup_{[0,T] \times [0,1]} |u^h - u|^p \right] &=   \E \left[\indicator_{ \Omega_{Z, h_0}}  \sup_{[0,T] \times [0,1]}  |u^h - u| ^p +\indicator_{ \Omega^c_{Z, h_0}}  \sup_{[0,T] \times [0,1]} |u^h - u|^p\right]\\
 \end{aligned}
\end{equation}
Note that by definition of the set $ \Omega_{Z, h_0}$, up to time $T$ the absolute value has not reached the value $Z$ yet, and as the $ \sup_{[0,T] \times [0,1]}$-norm goes only up to time $T$,  we can replace $u$ by the truncated solution $u_Z$  on the set $ \Omega_{Z, h_0}$.
Consequently, we can write for all $h < h_0$ 
\begin{equation}\label{int:step2}
\begin{aligned}
 \eqref{int:step1} &=   \E \left[\indicator_{ \Omega_{Z, h_0}} \sup_{[0,T] \times [0,1]} |\utr - u_Z|^p \right]+ \E \left[ \indicator_{ \Omega^c_{Z, h_0}}  \sup_{[0,T] \times [0,1]} |u^h - u|^p\right]\\
 \end{aligned}
\end{equation}
As by virtue of  Proposition \ref{prop:moments} and  Proposition \ref{prop:continuoussolution} 
\begin{equation}
\begin{aligned}
 \sup_h   \E \left[  \sup_{[0,T] \times [0,1]} |u^h - u|^{2p}\right]  &\leq  \sup_h \E \left[ \sup_{[0,T] \times [0,1]} |u^h(x,t)|^{2p}\right] + \E \left[ \sup_{[0,T] \times [0,1]} |u(x,t)|^{2p}\right] \\
  &\leq C (\gamma,p,T) \\
  \end{aligned}
\end{equation}
the second term in \eqref{int:step2} can be reformulated with Cauchy-Schwarz  
\begin{equation}
\begin{aligned}
 \E \left[ \indicator_{ \Omega^c_{Z, h_0}} \sup_{[0,T] \times [0,1]} |u^h - u|^p \right] \; &\leq \; \P\left[\Omega^{c}_{Z, h_0}\right]^{1/2} 
 \E \left[  \sup_{[0,T] \times [0,1]} |u^h - u|^{2p}\right]^{1/2}\\
  &\leq \; c (\gamma,p,T) \cdot \P\left[\Omega^{c}_{Z, h_0}\right]^{1/2} 
 \end{aligned}
\end{equation}
Employing Proposition \ref{prop:truncatedconv}, equation \eqref{int:step2} reads
 \begin{equation}\label{int:step3}
\begin{aligned}
  \E \left[ \sup_{[0,T] \times [0,1]} |u^h - u|^p \right] 
  &\overset{\eqref{int:truncatedconv}}{\leq}  c (\gamma,p,T)   h^{\frac{p}{2} - \delta p }   + C (p,T) \cdot \P\left[\Omega^{c}_{Z, h_0}\right]^{1/2} \\
 \end{aligned}
\end{equation}
\eqref{int:OmegaZ} implies that we can choose some $h_0$ such that $\P \left[ \Omega^c_{Z,h_0} \right] \leq 2 \epsilon$, and so 
 \begin{equation}\label{int:step4}
\begin{aligned}
  \E \left[ \sup_{[0,T] \times [0,1]} |u^h - u|^p \right]  \; \leq \;  c (\gamma,p,T)  h^{\frac{p}{2} - \delta p}   + C (p,T) \sqrt{ \epsilon} \\
 \end{aligned}
\end{equation}
But $ \epsilon$ was arbitrary, so we take the $p$-th root and conclude 
 \begin{equation}\label{int:step5}
\begin{aligned}
 \limsup_{h \to 0} \E \left[  \sup_{[0,T] \times [0,1]} |u^h - u|^p \right]  \;  \leq \; c (\gamma,p,T)  h^{\frac{1}{2} - \delta}   \\
 \end{aligned}
\end{equation}
which converges to zero if  $2-2\zeta > 0$. However, the usage of Proposition \ref{prop:sgconv} restricts us to $\zeta < \frac{1}{2}$.
\end{proof}

\section{Convergence of transition times}
Solutions to equation \eqref{int:generalac} have a metastable behaviour: After spending a long time close to one minimum of the double well potential $V$, they switch quickly to the other minimum. The time which the solution needs to make this transition is called the \textit{transition time}. 

This section adds the observation that the transition times  $ \tau^h (B)$ for the discrete system \eqref{int:discretesystem} converge to the transition times for the limit equation \eqref{int:generalac}, of which precise estimates have been proved in \cite{barret} via a nearest-neighbour approximation scheme.   

\bigskip

Given the initial condition $u_0$ close to one minimum of the potential $V$. We fix a function $u_{\textup{min}} \in C([0,1])$ close to the other minimum and take a ball of radius $\rho$ around $u_{\textup{min}}$. We want to estimate the time that the trajectory of our solution $u(t)$ enters this neighbourhood of $u_{\textup{min}}$ for the first time. 
As a measure of distance from $u_{\textup{min}}$, we take the open ball in $L^q([0,1])$.

We define 
\begin{equation}
 \tau (\rho , q) := \inf \left\lbrace t > 0, \|u(t) - u_{\textup{min}}\|_{L^q([0,1])} < \rho \right\rbrace
\end{equation}
\begin{equation}
 \tau^h (\rho , q) := \inf \left\lbrace t > 0, \|u(t)^h - u_{\textup{min}}^h\|_{L^q([0,1])} < \rho \right\rbrace
\end{equation}

\begin{lemma}\label{lemma:tau}
Under the conditions of  Theorem \ref{theo:asconvergence}, for all $\rho >0$, 
 \begin{equation}
 \tau(\rho, q) \; \leq \; \liminf_{h\to 0}  \tau^h(\rho , q) \; \leq \; \limsup_{h\to 0}  \tau^h(\rho , q) \; \leq \;   \tau(\rho^- , q) \qquad{ a. s.}
 \end{equation}
where $ \tau(\rho^- , q) = \lim_{\delta \to 0^+} \tau(\rho - \delta , q)$
\end{lemma}
\begin{proof}{Lemma}{lemma:tau}
Theorem \ref{theo:asconvergence} implies the almost sure convergence of $u^h$ to $u$ in $C([0,T], L^q([0,1]))$. Consequently, there exists $h_0(\omega)$ such that for all $h \leq h_0(\omega)$
\begin{equation}
 \sup_{t\in [0,T]}\left[ \int_0^1 |u^h - u|^q \, dx \right]^{\frac{1}{q}}(\omega) \; \leq \; \frac{\delta}{2}
\end{equation}
and
\begin{equation}
\| u_{\textup{min}}^h - u_{\textup{min}}\|_{L^q([0,1])}  \; \leq \; \frac{\delta}{2}
\end{equation}
We see that 
\begin{equation}
 \| u(t) - u_{\textup{min}}\|_{L^q([0,1])}  \geq  \rho + \delta
\end{equation}
and calculate
\begin{equation}
\begin{aligned}
\| u(t) - u_{\textup{min}}\|_{L^q([0,1])} 
&\leq \| u(t) - u^h(t)\|_{L^q([0,1])} + \| u^h(t) - u^h_f\|_{L^q([0,1])} + \| u^h_f(t) - u_{\textup{min}}\|_{L^q([0,1])}\\
&\leq \; \delta +  \| u^h(t) - u^h_f\|_{L^q([0,1])}
\end{aligned}
\end{equation}
which gives 
$ \min \lbrace  \tau^h(\rho , q), T \rbrace \geq t  $
and 
\begin{equation}\label{int:lowertau}
  \min \lbrace  \tau(\rho + \delta , q), T \rbrace \leq \liminf_{h\to 0} \min \lbrace  \tau^h(\rho , q), T \rbrace
\end{equation}
Similarly, we have for $t \leq   \min \lbrace  \tau^h(\rho , q), T \rbrace$ and all $h \leq h_0(\omega)$
\begin{equation}
\begin{aligned}
\| u^h(t) - u^h_f\|_{L^q([0,1])} 
&\leq \| u(t) - u^h(t)\|_{L^q([0,1])} + \| u(t) - u_{\textup{min}}\|_{L^q([0,1])} + \| u^h_f(t) - u_{\textup{min}}\|_{L^q([0,1])}\\
&\leq \; \delta +  \| u(t) - u_{\textup{min}}\|_{L^q([0,1])}
\end{aligned}
\end{equation}
from which we conclude
\begin{equation}\label{int:uppertau}
\limsup_{h\to 0} \min \lbrace  \tau^h(\rho , q), T \rbrace \; \leq \; \min \lbrace  \tau(\rho - \delta , q), T \rbrace 
\end{equation}
Therefore,
 \begin{equation}
 \tau(\rho + \delta , q) \; \leq \; \liminf_{h\to 0}  \tau^h(\rho , q) \; \leq \; \limsup_{h\to 0}  \tau^h(\rho , q) \; \leq \;   \tau(\rho - \delta , q)
 \end{equation}
As $ \rho \mapsto \tau(\rho , q)$ and $ \rho \mapsto \tau^h(\rho , q)$ are cadlag functions, we can deduce
 \begin{equation}
 \tau(\rho, q) \; \leq \; \liminf_{h\to 0}  \tau^h(\rho , q) \; \leq \; \limsup_{h\to 0}  \tau^h(\rho , q) \; \leq \;   \tau(\rho^- , q) \qquad{ a. s.}
 \end{equation}
\end{proof}
 
 \begin{theorem}\label{theo:transition}
 Under the conditions of  Theorem \ref{theo:asconvergence}, for almost all $\rho >0$,
 \begin{equation}
 \tau^h(\rho , q)  \longrightarrow \tau(\rho , q)  \qquad  \textup{a. s.  as } h \to 0
 \end{equation}
 and 
  \begin{equation}
 \E\left[\tau^h(\rho , q) \right] \longrightarrow  \E\left[ \tau(\rho , q) \right]   \qquad \textup{a. s.  as } h \to 0
 \end{equation}
 \end{theorem}
\begin{proof}{Theorem}{theo:transition}
The proof is very similar to the nearest case treated in \cite{barret}.
 Lemma \ref{lemma:tau} shows the statement for all points of continuity of $\rho \mapsto \tau(\rho , q)$. As $\rho \mapsto \tau(\rho , q)$ is cadlag and increasing on almost all $\omega \in \Omega$, there are  at most countably many points of discontinuity. 
 
 Let us denote by $ \mathscr{D}$ the points of discontinuity for a fixed $\omega$ and by $\mathscr{J} $ the jump set:
 \begin{equation}
 \mathscr{D}(\omega ) = \left\lbrace \rho \in \real^+, \tau (\rho^-,q) \neq \tau (\rho,q) \right\rbrace \qquad  \mathscr{J} = \bigcup_{\omega} \lbrace \omega \rbrace \times \mathscr{D}(\omega)
 \end{equation}
 For each $\omega$, there are at most countably many points of discontinuity, i.e. $ \mathscr{D}(\omega )$ has Lebesgue measure zero. 
 
 Let us now  call $\mathscr{N}$ the set of $\omega$ for which $\tau$ is discontinuous in $\rho$. This gives us an alternative description of the jump set in terms of unions of $\rho$:
  \begin{equation}
 \mathscr{J} = \bigcup_{\omega} \lbrace \omega \rbrace \times \mathscr{D}(\omega) \; = \; \bigcup_{\rho \in \real^+_{>0} } \mathscr{N}(\rho) \times \lbrace \rho \rbrace
 \end{equation}
We calculate with Fubini
\begin{equation}
\begin{aligned}
\int_{\real^+_{>0} } \P[\mathscr{N}(\rho)] d\rho \; = \; \int_{\real^+_{>0} } \E[\indicator_{\mathscr{N}(\rho)}] d\rho 
\; &= \; \int_{\Omega} \int_{\real^+_{>0} } \indicator_{\mathscr{J}(\omega \rho)} d\rho \, d\P(\omega ) \\
 &= \int_{\Omega \setminus \mathscr{N}} \int_{\real^+_{>0} } \indicator_{\mathscr{D}(\omega)}( \rho) d\rho \, d\P(\omega ) =0
\end{aligned}
\end{equation}
Therefore, there exists a set of measure zero $A$ of $\real^+$
\begin{equation}
\P[\mathscr{N}(\rho)]  = 0 \qquad \textup{ for all } \rho \in \real^+ \setminus A
\end{equation} 
 Consequently, $ \tau (\rho^-,q)(\omega ) =  \tau (\rho,q) (\omega )$ for all  $\rho \in \real^+ \setminus A$ and  $\omega \in \Omega \setminus \mathscr{N}(\rho)$.
 This means, as $h \to 0$, 
 \begin{equation}
  \tau^h(\rho , q)  \longrightarrow \tau(\rho , q)  \qquad \omega-\textup{a. s.} , \rho-\textup{ a.e.} 
 \end{equation}
 By dominated convergence, we conclude that for all $\rho \in \real^+ \setminus  \mathscr{D}$
   \begin{equation}
 \E\left[\tau^h(\rho , q) \right] \longrightarrow  \E\left[ \tau(\rho , q) \right]   \qquad \textup{ as } h \to 0.
 \end{equation}
\end{proof}

\bibliographystyle{plain}

\end{document}